\documentclass[a4paper,11pt,reqno,oneside]{amsart}

\usepackage{lmodern}
\usepackage[x11names]{xcolor}
\usepackage[english]{babel}
\usepackage[shortlabels]{enumitem}

\usepackage{amsmath}
\usepackage{amsfonts}
\usepackage{amssymb}
\usepackage{amsthm}
\usepackage{amscd}
\usepackage{amsaddr}

\usepackage{marvosym} 
\usepackage{mathtools} 
\usepackage{stmaryrd} 
\usepackage{csquotes}

\usepackage[backend=biber, giveninits, style=alphabetic, uniquelist=false, natbib=true, doi=false, url=false, isbn = false, bibencoding=utf8]{biblatex}
\usepackage{bib}
\usepackage{graphicx}
\usepackage{abbrev}
\usepackage[paper=a4paper,left=30mm,right=30mm,top=30mm,bottom=30mm]{geometry}
\usepackage{hyperref}
\hypersetup{colorlinks=true,urlcolor=blue,linkcolor=blue,citecolor=blue,naturalnames=true,hypertexnames=false}

\title[Inverse intensity estimation]{Nonparametric intensity estimation from noisy observations of a Poisson process under unknown error distribution}

\author{Martin Kroll}
\thanks{Resarch for this article was performed while I was PhD student at Universit{\"a}t Mannheim. Financial support by the Deutsche Forschungsgemeinschaft (DFG) through the Research Training Group RTG 1953 is gratefully acknowledged. I am indepted to my supervisors Jan Johannes and Martin Schlather for fruitful discussions and helpful comments on the paper.}
\address{ENSAE-ParisTech CREST}
\email{martin.kroll@ensae.fr}

\date{\today}

\begin{document}

\begin{abstract}
We consider the nonparametric estimation of the intensity function of a Poisson point process in a circular model from indirect observations $N_1,\ldots,N_n$.
These observations emerge from hidden point process realizations with the target intensity through contamination with additive error.
In case that the error distribution can only be estimated from an additional sample $Y_1,\ldots,Y_m$
we derive minimax rates of convergence with respect to the sample sizes $n$ and $m$ under abstract smoothness conditions and propose an orthonormal series estimator which attains the optimal rate of convergence.
The performance of the estimator depends on the correct specification of a dimension parameter whose optimal choice relies on smoothness characteristics of both the intensity and the error density.
We propose a data-driven choice of the dimension parameter based on model selection and show that the adaptive estimator attains the minimax optimal rate.
\end{abstract}

\maketitle

\section{Introduction}

Point process models are used in a wide variety of applications, including, amongst others, stochastic geometry~\cite{stoyan2013stochastic}, extreme value theory~\cite{resnick1987extreme}, and queueing theory~\cite{bremaud1981point}.
Each realization of a point process is a random set of points $\{ x_j \}$ which can alternatively be represented as an $\N_0$-valued random measure $\sum_{j} \delta_{x_j}$ where $\delta_\bullet$ denotes the Dirac measure concentrated at $\bullet$.
Poisson point processes (PPPs) are of particular importance since they serve as the elementary building blocks for more complex point process models.
Let $\X$ be a locally compact second countable Hausdorff space, $\Xs$ the corresponding Borel $\sigma$-field and $\Lambda$ a locally finite measure on the measurable space $(\X,\Xs)$, i.e., $\Lambda(C) < \infty$ for all relatively compact sets $C$ in $\Xs$.
A random set of points $N=\{ x_j \}$ from $\X$ (resp. the random measure $N=\sum_{j} \delta_{x_j}$) is called \emph{Poisson point process} with \emph{intensity measure} $\Lambda$ if \emph{(i)} the number $N_C = \vert N \cap C \vert$ of points located in $C$ follows a Poisson distribution with parameter $\Lambda(C)$ for all relatively compact $C \in \Xs$, and \emph{(ii)} for all $n \in \N$ and disjoint sets $A_1,\ldots,A_n \in \Xs$, the random variables $N_{A_1},\ldots,N_{A_n}$ are independent. 
It is well-known that the distribution of a PPP is completely determined by its intensity measure.
Hence, from a statistical point of view, the (nonparametric) estimation of the intensity measure or its Radon-Nikodym derivative (the \emph{intensity function}) with respect to some dominating measure from observations of the point process is of fundamental importance.

Inference and testing problems for Poisson and more general point processes have been tackled in a wide range of scenarios. The monographs~\cite{karr1991point} and~\cite{kutoyants1998statistical} offer a comprehensive overview and discuss both parametric and nonparametric methods.
From a methodological point of view, our approach in this paper is related to the article~\cite{reynaud2003adaptive} where the estimation of the intensity function from \emph{direct} observations was studied using concentration inequalities.

Other approaches to nonparametric intensity estimation from direct observations, without making a claim to be exhaustive, can be found in~\cite{baraud2009estimating} (where the performance of a histogram estimator under Hellinger loss is analysed),~\cite{birge2007model} (using a testing approach to model selection),~\cite{gregoire2000convergence} (using a minimum complexity estimator in the Aalen model), and~\cite{patil2004counting} (suggesting a wavelet estimator in the multiplicative intensity model).

Theoretical work on intensity estimation has recently been motivated by applications to genomic data.
The model considered in the article~\cite{bigot2013intensity} is motivated by data arising throughout the processing of DNA ChIP-seq data.
The article~\cite{sansonnet2014wavelet} takes its motivation from the analysis of genomic data as well.
In addition, let us mention two further articles where the development of nonparametric statistical methods for the analysis of point processes was inspired by applications from biology:
first, motivated through DNA sequencing techniques, the article~\cite{shen2012change} introduces a change-point model for nonhomogeneous Poisson processes occurring in molecular biology.
Second, the article~\cite{zhang2010nonparametric} considered the nonparametric inference of Cox process data by means of a kernel type estimator.

Usually one aims to estimate the intensity function $\lambda$ from \emph{direct} observations $\widetilde N_1,\ldots,\widetilde N_n$ where 
\begin{equation}\label{eq:gen:hidden}
\widetilde N_i = \sum_{j} \delta_{x_{ij}}
\end{equation}
are realizations of a PPP with the target intensity $\lambda$.
In this paper, however, we assume that we are interested in the nonparametric estimation of the intensity function $\lambda$ without having access to the observations in~\eqref{eq:gen:hidden}.
Instead, we are in the setup of a \emph{Poisson inverse problem}~\cite{antoniadis2006poisson} where we can only observe $N_1,\ldots,N_n$ given through
\begin{equation}\label{eq:gen:obs}
N_i = \sum_{j} \delta_{y_{ij}}.
\end{equation}
The \emph{indirect} observations $N_i$ are related to the hidden $\widetilde N_i$ by the identity $y_{ij} = x_{ij} + \epsilon_{ij} - \lfloor x_{ij} + \epsilon_{ij} \rfloor$.
The definition of the $y_{ij}$ as the fractional part of the additively contaminated $x_{ij}$ yields a circular model by means of the usual topological identification of the interval $[0,1)$ and the circle of perimeter $1$.

In contrast to our approach, the few existing papers on Poisson inverse problems \cite{cavalier2002poisson,antoniadis2006poisson,bigot2013intensity} assume the error distribution to be known.
This conservative assumption is also standard in research articles dealing with classical deconvolution problems~\cite{meister2009deconvolution}.
If the error density is unknown, even identifiability of the statistical model is not guaranteed.
Several remedies have been introduced to overcome this problem:
for instance, it is possible to impose additional assumptions on the statistical model (e.g.,~\cite{schwarz2010consistent} which deals with blind convolution under additive Gaussian noise with unknown variance).
Alternatively, one can consider a framework with panel data~\cite{neumann2007deconvolution}.
Finally, one can assume the availability of an additional sample from the error density (e.g.,~\cite{diggle1993fourier,johannes2009deconvolution,comte2010pointwise,comte2011data-driven}) to guarantee identifiability and enable inference.
In this paper, we will stick to this last option.

Let us assume that the errors $\epsilon_{ij}$ in the general model~\eqref{eq:gen:obs} are i.i.d.\ $\sim f$ for some \emph{unknown} error density $f$.
We will study the resulting model and consider the nonparametric estimation of the intensity function from observations
\begin{equation}\label{eq:obs}
N_1,\ldots,N_n  \quad \text{i.i.d.} \qquad \text{and} \qquad Y_1,\ldots,Y_m \quad \text{i.i.d.} \sim f
\end{equation}
where the $N_i$ are given as in~\eqref{eq:gen:obs}.
A natural aim here is to detect optimal rates of convergence in terms of the sample sizes $n$ and $m$ and to construct adaptive estimators attaining these rates.
Note that the observation of $n$ i.i.d.\,processes $N_1,\ldots,N_n$ with intensity $\lambda$ is equivalent to the observation of one process $N$ with intensity $n\lambda$, and both directions of this equivalence can easily be made rigorous.
In order to obtain $N$ from the $N_1,\ldots,N_n$, put $N = N_1 \cup \ldots \cup N_n$ (denoting by $\cup$ the set-theoretic union of point processes; this shows the infinite divisibility of Poisson point processes).
For the other direction, given $N$ and $n$, it suffices to assign every point $x \in N$ to one of the processes $N_1,\ldots,N_n$ with equal probability.

From a methodological point of view, our approach is inspired by the one conducted in~\cite{johannes2013adaptive}.
We consider orthonormal series estimators of the form
\begin{equation*}\label{eq:ons:fou:est}
\widehat \lambda_k = \sum_{0 \leq \vert j \vert \leq k} \widehat{\fou{\lambda}}_j \e_j
\end{equation*}
where $\e_j(\cdot) = \exp(2\pi ij \cdot)$ and $\widehat{\fou{\lambda}}_j$ is an appropriate estimator of the Fourier coefficient ${\fou{\lambda}}_j$ corresponding to the basis function $\e_j(\cdot)$ (see Section~\ref{s:meth} for details). Of course, this estimator is motivated by the $\LL^2$-convergent representation $\lambda = \sum_{j \in \Z} \fou{\lambda}_j \e_j$ for square-integrable $\lambda$.
It turns out that the performance of the estimator $\widehat \lambda_k$ crucially depends on the choice of the dimension parameter $k$ and that its optimal value depends on smoothness characteristics of the intensity that are usually not available in practice.
In order to choose $k$ in a completely data-driven manner, we follow an approach based on model selection (see \cite{barron1999risk,comte2015estimation}) and select the dimension parameter as the minimizer of a penalized contrast criterion.
For the theoretical analysis of the adaptive estimator we need Talagrand type concentration inequalities tailored to the framework with PPP observations which cannot be directly transferred from results applied in the usual density estimation or deconvolution frameworks (see Remark~2.2 in~\cite{kroll2016concentration}).
These inequalities have already been derived in a separate manuscript~\cite{kroll2016concentration}, and we only state the necessary consequences of these results in the appendix.
The article is organized as follows: in Section~\ref{s:meth} we introduce our methodological approach.
In Section~\ref{S:MINIMAX} we study the nonparametric estimation problem from a minimax point of view.
Section~\ref{S:ADAP} considers adaptive estimation of the intensity for the Poisson model.
Proofs are given in Section~\ref{s:proofs:minimax} and~\ref{s:proofs:adap:poisson}.
\section{Methodology}\label{s:meth}
\subsection{Notation}

Throughout this work we assume that the intensity $\lambda$ and the density $f$ belong to the space $\LL^2 = \LL^2([0,1), \dd x)$ of square-integrable functions on the interval $[0,1)$.
Let $\{\e_j \}_{j \in \Z}$ be the \emph{complex trigonometric basis} of $\LL^2$ given by $\e_j(t) =\exp(2\pi i j t)$.
The Fourier coefficients of a function $g \in \LL^2$ are denoted as follows:
\begin{equation*}
\fou{g}_j = \int_0^1 g(t) \e_j(-t)\dd t.
\end{equation*}
For a strictly positive symmetric sequence $\omega=\seqomega$ we introduce the weighted norm $\lVert \cdot \rVert_\omega$ defined via $\lVert g \rVert_\omega^2 = \sum_{j \in \Z} \omega_j \lvert \fou{g}_j \rvert^2$.
The corresponding scalar product is denoted with $\langle g,h \rangle_\omega = \sum_{j \in \Z} \omega_j \fou{g}_j \conj{\fou{h}}_j$.
Throughout the paper, we use the notation $a(n,m) \lesssim b(n,m)$ if $a(n,m) \leq C \cdot b(n,m)$ for some numerical constant $C$ independent of $n$ and $m$.
\subsection{The minimax point of view}
We evaluate the performance of an arbitrary estimator $\widetilde \lambda$ of $\lambda$ by means of the mean integrated weighted squared loss $\E [\Vert \widetilde \lambda - \lambda\Vert_\omega^2]$.
We take up the minimax point of view and consider the \emph{maximum risk} defined by
\begin{equation}\label{eq:risk}
	\sup_{\lambda \in \Lambda} \sup_{f \in \Fc} \E [\Vert \widetilde \lambda - \lambda\Vert_\omega^2]
\end{equation}
where $\Lambda$ and $\Fc$ are classes of potential intensity functions $\lambda$ and densities $f$, respectively.
The \emph{minimax risk} is defined via
\begin{equation*}
	\inf_{\widetilde \lambda} \sup_{\lambda \in \Lambda} \sup_{f \in \Fc} \E [\Vert \widetilde \lambda - \lambda\Vert_\omega^2]
\end{equation*}
where the infimum is taken over all estimators $\widetilde \lambda$ of $\lambda$.
An estimator $\lambda^\ast$ is called \emph{rate optimal} if
\begin{equation*}
	\sup_{\lambda \in \Lambda} \sup_{f \in \Fc} \E [\Vert \lambda^\ast - \lambda\Vert_\omega^2]	\lesssim \inf_{\widetilde \lambda} \sup_{\lambda \in \Lambda} \sup_{f \in \Fc} \E [\Vert \widetilde \lambda - \lambda\Vert_\omega^2].
\end{equation*}
By allowing for general weight sequences $\omega$, we can treat both the estimation of $\lambda$ (in this case, $\omega \equiv 1$) as well as the estimation of derivatives (take $\omega_j = j^{2s}$ for $\vert j \vert \geq 1$ for the $s$-th derivative).
The classes $\Lambda$ of intensity functions and $\Fc$ of densities to be considered in this article will be specified in Section~\ref{S:MINIMAX} below where we derive lower bounds on the minimax risk for these specific choices and prove that this lower bound is attained up to a numerical constant by a suitably defined orthonormal series estimator.

\subsection{Sequence space representation}

Under the considered model, the observed point processes $N_1,\ldots,N_n$ in~\eqref{eq:obs} are generated from independent Poisson point processes $\widetilde N_1,\ldots,\widetilde N_n$ with intensity function $\lambda$ by independent random contaminations of the individual points.
We emphasize again that the (unobserved) contaminations are assumed to follow a probability law given by an unknown density $f$ and are to be understood additively modulo $1$.
Thus, the observations $N_i$ under the Poisson model are given by
\begin{equation*}
	N_i = \sum_j \delta_{x_{ij} + \epsilon_{ij} - \lfloor x_{ij} + \epsilon_{ij} \rfloor} 
\end{equation*}
where $\widetilde N_i = \sum_j \delta_{x_{ij}}$ is the realization of a Poisson point process with intensity function $\lambda$ and the errors $\epsilon_{ij}$ are i.i.d. $\sim f$.
Note that each $N_i$ is again a realization of a Poisson point process whose intensity function is given by the circular convolution $\lambda \star f$ modulo 1 of $\lambda$ with the error density $f$. More precisely, $\ell = \lambda \star f$ is given by the formula

\begin{equation*}\label{eq:ell}
	\ell(t) = \int_0^1 \lambda((t-\epsilon) - \lfloor t-\epsilon \rfloor) f(\epsilon)\dd \epsilon, \qquad t \in [0,1).
\end{equation*}
By the convolution theorem, we have $\fou{\ell}_j = \fou{\lambda}_j \cdot \fou{f}_j$ for all $j \in \Z$.
From Campbell's theorem (cf.~\cite{serfozo2009basics}, Chapter~3, Theorem~24) it can be deduced that for measurable functions $g$ we have
$\E [ \int_0^1 g(t)\dd N_i(t) ] = \int_0^1 \ell(t)g(t) \dd t$
provided that the integral on the right-hand side exists.
Exploiting this equation for $g(t)=\e_j(-t)$ and setting
\begin{equation}\label{eq:def:ellhat}
	\widehat{\fou{\ell}}_j = \frac 1 n \sum_{i=1}^n \int_0^1 \e_j(-t)\dd N_i(t)
\end{equation}
we thus obtain that $\E \widehat{\fou{\ell}}_j   = \fou{\lambda}_j \cdot \fou{f}_j$ for all $j \in \Z$.
More precisely, we have
\begin{equation}\label{eq:seq:I}
	\widehat{\fou{\ell}}_j = \fou{\lambda}_j\cdot \fou{f}_j + \xi_j \qquad \text{for all } j \in \Z
\end{equation}
where $\xi_j=\widehat{\fou{\ell}}_j - \E  \widehat{\fou{\ell}}_j =\frac 1 n \sum_{i=1}^{n} [\int_0^1 \e_j(-t)\dd N_i(t) - \int_0^1 \ell(t)\e_j(-t)\dd t ].$

\subsection{Orthonormal series estimator}
In view of~\eqref{eq:seq:I} and the fact that $\E \xi_j = 0$, a natural estimator of $\lambda$ is given by
\begin{equation}\label{eq:est}
	\widehat \lambda_k = \sum_{0 \leq \abs{j} \leq k} \frac{\widehat{\fou{\ell}}_j}{\widehat{\fou{f}}_j} \1_{\Omega_j}\e_j
\end{equation}
with $\widehat{\fou{\ell}}_j$ as defined in~\eqref{eq:def:ellhat}, $\widehat{\fou{f}}_j\defeq\frac{1}{m} \sum_{i=1}^{m} \e_j(-Y_i)$ and $\Omega_j\defeq \{ \vert \widehat{\fou{f}}_j \vert^2 \geq 1/m \}$.
Note that $\fou{f}_j$ in~\eqref{eq:seq:I} 
is not directly available and thus has to be estimated from the sample $Y_1,\ldots,Y_m$ in~\eqref{eq:obs}.
The additional threshold occurring in the definition of $\widehat \lambda_k$ through the indicator function over the set $\Omega_j$ compensates for 'too small' absolute values of $\widehat{\fou{f}}_j$ and is imposed in order to avoid unstable behaviour of the estimator.
The optimal choice $\knast$ of the dimension parameter in the minimax framework will be determined in Section~\ref{S:MINIMAX} and depends on the classes $\Lambda$ and $\Fc$.
The data-driven choice of the dimension parameter is discussed in Section~\ref{S:ADAP}.
\section{Minimax theory}\label{S:MINIMAX}

\subsection{Model assumptions}
Let $\gamma=(\gamma_j)_{j \in \Z}$ and $\alpha = (\alpha_j)_{j \in \Z}$ be strictly positive symmetric sequences and fix $r > 0, d \geq 1$.
In this section, we derive minimax rates of convergence concerning the maximum risk defined in~\eqref{eq:risk} with respect to the classes
\begin{equation*}
	\Lambda_\gamma^r \defeq \{ \lambda \in \LL^2 \colon \lambda \geq 0 \text{ and }   \sum_{j \in \Z} \gamma_j \vert \fou{\lambda}_j \vert^2 \eqdef \norm{\lambda}_\gamma^2 \leq r \}
\end{equation*}
and
\begin{equation*}
	\Fc_\alpha^d \defeq \{ f \in \LL^2 \colon f \geq 0, \fou{f}_0=1 \text{ and } d^{-1} \leq \abs{\fou{f}_j}^2 / \alpha_j \leq d \}
\end{equation*}
of intensity functions and error densities, respectively.
We now state some regularity conditions imposed on the sequences $\gamma$ and $\alpha$.

\begin{ass}\label{ass:seq}
	$\gamma = (\gamma_j)_{j \in \Z}$, $\alpha = (\alpha_j)_{j \in \Z}$ and $\omega = (\omega_j)_{j \in \Z}$ are strictly positive symmetric sequences such that $\gamma_0=\omega_0=\alpha_0=1$, $\gamma_j \geq 1$ for all $j \in \Z$ and the sequences $(\omega_n/\gamma_n)_{n \in \Nzero}$ and $(\alpha_n)_{n \in \Nzero}$ are both non-increasing.
	Finally, $\rho \defeq \sum_{j \in \mathbb Z} \alpha_j < \infty$.
\end{ass}

\subsection{Minimax lower bounds}

The following two theorems provide minimax lower bounds in terms of the sample sizes $n$ and $m$ in~\eqref{eq:obs}, respectively.
To state our results, we put
\begin{align*}
  \Psi_n &= \min_{k \in \N_0} \max \left\lbrace \frac{\omega_{k}}{\gamma_{k}}, \sum_{0 \leq \abs{j} \leq k} \frac{\omega_j}{n\alpha_j} \right\rbrace, \quad n \in \N,\\
  \Phi_m &= \max_{k \in \N} \left\lbrace \frac{\omega_k}{\gamma_k} \min \left(1,\frac{1}{m\alpha_k} \right) \right\rbrace, \quad m \in \N.
\end{align*}
By the results of this section, $\Psi_n$ and $\Phi_m$ will turn out to be the optimal (up to constants) rates of convergence  in terms of $n$ and $m$.
The two terms over which the maximum is taken in the definition of $\Psi_n$ can be interpreted as a squared bias term and a variance term, respectively.
The rate in $n$ should then be obtained by choosing the truncation value such that the maximum of these two terms is minimized.
This suggests to choose the truncation parameter as
\begin{equation}\label{EQ:DEF:KNAST}
k_n^\ast = \argmin_{k \in \N_0} \max \bigg\{ \frac{\omega_{k}}{\gamma_{k}}, \sum_{0 \leq \abs{j} \leq k} \frac{\omega_j}{n\alpha_j} \bigg\}.
\end{equation}
Our first theorem establishes a lower bound in terms of $n$.
\begin{thm}\label{THM:L:N}
	Let Assumption~\ref{ass:seq} hold, and further assume that
	\begin{enumerate}[label=(C\arabic*), leftmargin=1cm, itemsep=0em]
		\item\label{it:C1} $\Gamma\defeq\sum_{j \in \Z} \gamma_j^{-1} < \infty$, and
		\item $0 < \eta^{-1} = {\inf_{n \in \N} \Psi_n^{-1} \cdot \min \big\{ \frac{\omega_{k_n^\ast}}{\gamma_{k_n^\ast}}, \sum_{0 \leq \abs{j} \leq k_n^\ast} \frac{\omega_j}{n\alpha_j} \big\}}$
		for some $1 \leq \eta < \infty$.
	\end{enumerate}
	\noindent Then, for any $n \in \N$,
	\begin{equation*}
		\inf_{\widetilde \lambda} \sup_{\lambda \in \Lambda_\gamma^r} \sup_{f \in \Fc_\alpha^d} \E [ \Vert \widetilde \lambda- \lambda \Vert_\omega^2 ] \geq \frac{\zeta r}{16\eta} \cdot \Psi_n
	\end{equation*}
	where $\zeta=\min \{ \frac{1}{2\Gamma d \eta}, \frac{2\delta}{d\sqrt r} \}$ with $\delta = \frac{1}{2}-\frac{1}{2\sqrt 2}$, and the infimum is taken over all estimators $\widetilde \lambda$ of $\lambda$ based on the observations from~\eqref{eq:obs}.
\end{thm}
As the proof Theorem~\ref{THM:L:N} shows, the lower bound $\Psi_n$, which does not depend on the sample size of the auxiliary sample from the error density, is valid also in case of a known error density.
The potential deterioration of the overall rate of convergence in contrast to this case is introduced by the uncertainty concerning the error density $f$.
Since this uncertainty is quantified by the sample size $m$, one would expect a dependence of the lower bound on $m$ as well.
This intuition is made rigorous by means of the following theorem.

\begin{thm}\label{THM:L:M}
	Let Assumption~\ref{ass:seq} hold, and in addition assume that
	\begin{enumerate}[label=(C\arabic*), leftmargin=1.5cm]
		\setcounter{enumi}{2}
		\item\label{it:C3} there exists a density $f$ in $\Fc_\alpha^{\sqrt d}$ with $f \geq 1/2$.
	\end{enumerate}
	\noindent Then, for any $m \in \N$,
	\begin{equation*}\label{eq:lb:m}
	\inf_{\widetilde \lambda} \sup_{\lambda \in \Lambda_\gamma^r} \sup_{f \in \Fc_\alpha^d} \E [ \Vert \widetilde \lambda- \lambda \Vert_\omega^2 ] \geq \frac{1-\sqrt 3/2}{8} \cdot \zeta^2rd^{-1/2}\cdot \Phi_m
	\end{equation*}
	where $\zeta=\min \{ 1/(4\sqrt d), 1-d^{-1/4} \}$ and the infimum is taken over all estimators $\widetilde \lambda$ of $\lambda$ based on the observations from~\eqref{eq:obs}.
\end{thm}

The next corollary is an immediate consequence of Theorems~\ref{THM:L:N} and~\ref{THM:L:M}.

\begin{cor}
	Under the assumptions of Theorems~\ref{THM:L:N} and~\ref{THM:L:M}, for any $n,m \in \N$,
	\begin{align*}
	\inf_{\widetilde \lambda} \sup_{\lambda \in \Lambda_\gamma^r} \sup_{f \in \Fc_\alpha^d} \E [ \Vert \widetilde \lambda- \lambda \Vert_\omega^2 ] &\geq \max \bigg\{ \frac{\zeta r}{16\eta} \cdot \Psi_n, \frac{1-\sqrt 3/2}{8} \cdot \zeta^2rd^{-1/2}\cdot \Phi_m  \bigg\}\\
	&\geq \frac 1 2 \bigg\{ \frac{\zeta r}{16\eta} \cdot \Psi_n + \frac{1-\sqrt 3/2}{8} \cdot \zeta^2rd^{-1/2}\cdot \Phi_m  \bigg\}.
	\end{align*}
\end{cor}

Note that the contributions of the sample sizes $n$ and $m$ to the overall lower bound are separated from another, and the rate is determined by the maximum of $\Psi_n$ and $\Phi_m$.
This phenomenon has already been observed in the related problem of density estimation \cite{johannes2009deconvolution,comte2011data-driven,johannes2013adaptive} and other inverse problems with unknown operator \cite{delattre2012blockwise,johannes2013gaussian}.
In addition, it can be seen from the mere definition of $\Psi_n$ and $\Phi_m$ that the rate in terms of $m$ is always faster then the one in $n$.
Hence, as long as $m \geq n$, there is no deterioration in the rate in comparison to the setup with known error density (see Table~\ref{tab:rates} for a more detailed evaluation of the rates in some special cases).

\subsection{Upper bound}

Let us now establish an upper bound for the maximum risk in terms of $n$ and $m$ for the estimator $\widehat \lambda_k$ in~\eqref{eq:est} under a suitable choice of the dimension parameter $k$.
More precisely, the following theorem establishes an upper bound for the rate of convergence of $\widehat \lambda_{\knast}$ with $k_n^*$ defined in Equation~\eqref{EQ:DEF:KNAST}. 
Thus, due to the lower bound proofs in the preceding subsection it is shown that $\widehat \lambda_{k_n^*}$ attains the minimax rates of convergence in terms of the samples sizes $n$ and $m$. 
Note that this rate optimal choice $\knast$ of the dimension parameter does not depend on the sample size $m$ (recall Equation~\eqref{EQ:DEF:KNAST} for its definition, and note that none of the quantities appearing there depends on $m$).
The non-dependence of the rate-optimal smoothing parameter can also been observed in the related model of circular density deconvolution with unknown error density considered in \cite{johannes2013adaptive}.

\begin{thm}\label{THM:U:NM:I:II}
	Let Assumption~\ref{ass:seq} hold.
	Then, for any $n,m \in \N$,
	\begin{equation*}
	\sup_{\lambda \in \Lambda_\gamma^r} \sup_{f \in \Fc_\alpha^d} \E [ \Vert \widehat \lambda_\knast - \lambda \Vert_\omega^2 ] \lesssim \Psi_n + \Phi_m.
	\end{equation*}
\end{thm}

\subsection{Examples of convergence rates}\label{subs:minimax:rates}

Fixing $\omega_0= 1$, $\omega_j=\vert j \vert^{2s}$ for $j \neq 0$ and some $s \geq 0$, we consider specific choices of the sequences $\gamma$ and $\alpha$ and state the resulting rates with respect to both sample sizes $n$ and $m$. 

\paragraph{Choices for the sequence $\gamma$:} 
\begin{itemize}
	\item {\color{red}(pol)}: $\gamma_0=0$ and $\gamma_j = \vert j \vert^{2p}$ for all $j \neq 0$ and some $p \geq 0$. This corresponds to the case when the unknown intensity function belongs to some \emph{Sobolev space}.
	\item {\color{red}(exp)}: $\gamma_j = \exp(2p\vert j \vert)$ for all $j \in \Z$ and some $p \geq 0$. In this case, $\lambda$ belongs to some space of \emph{analytic functions}.
\end{itemize}
\paragraph{Choices for the sequence $\alpha$:} 
\begin{itemize}
	\item {\color{blue}(pol)}: $\alpha_0=0$ and $\alpha_j = \vert j \vert^{-2a}$ for all $j \neq 0$ and some $a \geq 0$. This corresponds to the case when the error density is \emph{ordinary smooth}.
	\item {\color{blue}(exp)}: $\alpha_j = \exp(-2a\vert j \vert)$ for all $j \in \Z$ and some $a \geq 0$. 
\end{itemize}

Table~\ref{tab:rates} summarizes the rates $\Psi_n$ and $\Phi_m$ for the different choices of $\gamma$ and $\alpha$.
The rates in terms of $n$ coincide formally with the classical rates for nonparametric inverse problems (see \cite{fan1991optimal,lacour2006rates}, for instance).
The rates in $m$ are of the same order as those that have already been obtained in the related model of (circular) density deconvolution with unknown error density in \cite{johannes2009deconvolution,comte2011data-driven,johannes2013adaptive}.
They can also be compared with the rates in the indirect Gaussian sequence model with partially known operator \cite{johannes2013gaussian}, which provides a benchmark model for a variety of nonparametric inverse problems.

\def\arraystretch{1.5}
\setlength{\tabcolsep}{6pt}
\begin{table}
	\centering
	\begin{tabular}{ccccc}
		$\gamma$ & $\alpha$ & $\Theta(\Psi_n)$ & $\Theta(\Phi_m)$ & Restrictions\\
		\toprule
		\textcolor{red}{(pol)} & \textcolor{blue}{(pol)} & $n^{-\frac{2(p-s)}{2p+2a+1}}$ & $m^{-\frac{(p-s) \wedge a}{a}}$ & $p \geq s$, $a > \frac 1 2$\\
		\arrayrulecolor{black}\hline
		\textcolor{red}{(exp)} & \textcolor{blue}{(pol)} & $(\log n)^{2s+2a+1} \cdot n^{-1}$ & $m^{-1}$ & $a > \frac 1 2$\\
		\hline
		\textcolor{red}{(pol)} & \textcolor{blue}{(exp)} & $(\log n)^{-2(p-s)}$ & $(\log m)^{-2(p-s)}$ & $p \geq s$\\
		\hline
		\textcolor{red}{(exp)} & \textcolor{blue}{(exp)} & $(\log n)^{2s} \cdot n^{-\frac{p}{p+a}}$ & \def\arraystretch{1}\begin{tabular}{cl}
			\footnotesize{$(\log m)^{2s} \cdot m^{-p/a} $} & \footnotesize{if $a \geq p$}\\
			\footnotesize{$m^{-1}$} & \footnotesize{if $a < p$}
		\end{tabular}\def\arraystretch{1.7} &
	\end{tabular}
	\caption{Exemplary rates of convergence for nonparametric intensity estimation. The rates are given in the framework of Theorems~\ref{THM:L:N},~\ref{THM:L:M} and~\ref{THM:U:NM:I:II} which impose the given restrictions. In all examples $\omega_0=1$, $\omega_j=\vert j \vert^{2s}$ for $j \neq 0$, whereas the choices (pol) and (exp) for the sequences $\gamma$ and $\alpha$ are explained in Section~\ref{subs:minimax:rates}.}\label{tab:rates}
\end{table}
\def\arraystretch{1}
\section{Adaptive estimation}\label{S:ADAP}

The estimator considered in Theorem~\ref{THM:U:NM:I:II} is obtained by specializing the estimator in~\eqref{eq:est} with the truncation parameter $\knast$.
This procedure suffers from the apparent drawback that the resulting estimator depends on the knowledge of the classes $\Lambda_\gamma^r$ and $\Fc_\alpha^d$.
In this section, we provide adaptive choices of the truncation parameter based on model selection (see \cite{barron1999risk,massart2007concentration} for comprehensive presentations in the context of nonparametric estimation).
The principal idea of model selection procedures consists in defining a truncation parameter $\widehat k$ in a fully data-driven way as the minimizer of a penalized empirical contrast,
\begin{equation*}
 \widehat k \defeq \argmin_{k \in \mathcal M_n} \big\{ \Upsilon_n(\widehat \lambda_k) + \pen_k \big\},
\end{equation*}
where $\Upsilon_n: \Sc_n \to \R$ is a contrast function with $\Sc_n$ being the linear subspace of $\LL^2$ spanned by the functions $\e_j(\cdot)$ for $j \in \{ -n,\ldots,n \}$, $\pen_k$ a (as a function of $k$) non-decreasing penalty that mimics the variance, and $\mathcal M_n$ a set of admissible values of $k$ (which represents the set of admissible models since each choice of $k$ corresponds to a finite dimensional model which is given by the functions spanned by the basis functions $\e_j(\cdot)$ with $j \in \{ -k,\ldots,k \}$).

In order to construct an adaptive estimator which does not require any \emph{a priori} knowledge of $\Lambda_\gamma^r$ and $\Fc_\alpha^d$, we proceed in two steps:
in the first step, we assume that $\Lambda_\gamma^r$ is unknown but $\Fc_\alpha^d$ known.
Hence, the overall estimation procedure (in particular, the definition of the penalty term) might still depend on the knowledge of the sequence $\alpha=(\alpha_j)_{j \in \Z}$.
This results in a \emph{partially adaptive} definition $\widetilde k$ of the truncation parameter. 
In the second step, we dispense with any knowledge on the classes $\Lambda_\gamma^r$ and $\Fc_\alpha^d$ and propose a \emph{fully data-driven} choice $\widehat k$ of the truncation parameter.

\subsection{Partially adaptive estimation ($\Lambda_\gamma^r$ unknown, $\Fc_\alpha^d$ known)}\label{sub:partially}

For the definition of our partially adaptive choice of the dimension parameter we introduce some notation:
for any $k \in \N_0$, let
\begin{equation*}
\Delta_k^\alpha = \max_{0 \leq j \leq k} \omega_j \alpha_j^{-1} \quad \text{and} \quad \delta_k^\alpha = (2k+1) \Delta_k^\alpha \frac{\log(\Delta_k^\alpha \vee (k+3))}{\log (k+3)}.
\end{equation*}
For all $n, m  \in \N$, setting $\omega_j^+ = \max_{0 \leq i \leq j} \omega_i$, we define
\begin{align*}
&N_n^\alpha = \inf \left\lbrace 1 \leq j \leq n : \frac{\alpha_j}{2j+1} < \frac{\log(n+3) \omega_j^+}{n} \right\rbrace - 1 \wedge n,\\
&M_m^\alpha = \inf \{ 1 \leq j \leq m : \alpha_j < 640dm^{-1} \log(m+1) \} -1 \wedge m,
\end{align*}
and set $K_{nm}^\alpha = N_n^\alpha \wedge M_m^\alpha$.
Now, for $t \in \LL^2$, define the contrast $\Upsilon(t) = \Vert t \Vert_\omega^2 - 2 \Re \langle \widehat \lambda_{n \wedge m},t \rangle_\omega$
and the random sequence of penalties $(\tildepen_k)_{k \in \N_0}$ via $$\tildepen_k = \frac{165}{2} d\cdot (\widehat{\fou{\ell}}_0 \vee 1) \cdot \frac{\delta_k^\alpha}{n}.$$
Building on our definition of contrast and penalty, we define the partially adaptive selection of the dimension parameter $k$ as
\begin{equation*}
\kpart = \argmin_{0 \leq k \leq K_{nm}^\alpha} \{ \Upsilon(\widehat \lambda_k) + \tildepen_k\}.
\end{equation*}
\begin{thm}\label{THM:PART:ADAP} Let Assumption~\ref{ass:seq} hold.
	Then, for any $n,m \in \N$,
	\begin{equation*}
	\sup_{\lambda \in \Lambda_\gamma^r} \sup_{f \in \Fc_\alpha^d} \E [ \Vert \widehat \lambda_{\widetilde k} -\lambda \Vert_\omega^2 ] \lesssim \min_{0 \leq k \leq K_{nm}^\alpha} \max \left\lbrace \frac{\omega_k}{\gamma_k}, \frac{\delta_k^\alpha}{n} \right\rbrace  +  \Phi_m
	+  \frac{1}{m} + \frac{1} {n}.
	\end{equation*}
\end{thm}

\subsection{Fully adaptive estimation ($\Lambda_\gamma^r$ and $\Fc_\alpha^d$ unknown)}\label{sub:fully}

We now also dispense with the knowledge of the smoothness of the error density and propose a fully data-driven selection $\kfull$ of the dimension parameter. 
As in the case of partially adaptive estimation, we have to introduce some notation first.
For $k \in \N_0$, let
\begin{equation*}
\widehat \Delta_k = \max_{0 \leq j \leq k} \frac{\omega_j}{\vert \widehat{\fou{f}}_j\vert^2} \1_{\Omega_j} \qquad \text{and} \qquad \widehat \delta_k = (2k+1) \widehat \Delta_k \frac{\log(\widehat \Delta_k \vee (k+4))}{\log(k+4)}.
\end{equation*}
For $n,m \in \N$, set
\begin{align*}
&\widehat N_n = \inf \{ 1 \leq j \leq n : \vert \widehat{\fou{f}}_j \vert^2/(2j+1) < \log(n+4)\omega_j^+/n \} - 1 \wedge n,\\
&\widehat M_m = \inf \{1 \leq j \leq m : \vert \widehat{\fou{f}}_j\vert^2 <  m^{-1} \log(m)\} - 1 \wedge m,
\end{align*}
and $\Khatnm = \widehat N_n \wedge \widehat M_m$.
We consider the same contrast function as in the partially adaptive case but define the random sequence $(\hatpen_k)_{k \in \Nzero}$ of penalities now by $$\hatpen_k =  2750 \cdot (\widehat{\fou{\ell}}_0 \vee 1) \cdot \frac{\widehat \delta_k}{n}$$
which does no longer depend on $\alpha$ nor $d$.
Finally, set
\begin{equation*}
\kfull = \argmin_{0 \leq k \leq \Khatnm} \{ \Upsilon(\widehat \lambda_k) + \hatpen_k \}.
\end{equation*}
In order to state and prove the upper risk bound of the estimator $\widehat \lambda_{\widehat k}$, we have to introduce some further notation.
We keep the definition of $\Delta_k^\alpha$ from Subsection~\ref{sub:partially} but slightly redefine $\delta_k^\alpha$ as
\begin{equation*}
\delta_k^\alpha = (2k+1) \Delta_k^\alpha \frac{\log(\Delta_k^\alpha \vee (k+4)) }{\log(k+4)}.
\end{equation*}
For $k \in \Nzero$, we also define
\begin{equation*}
\Delta_k = \max_{0 \leq j \leq k} \frac{\omega_j}{\abs{\fou{f}_j}^2} \qquad \text{and} \qquad \delta_k = (2k+1) \Delta_k \frac{\log(\Delta_k \vee (k+4)) }{\log(k+4)},
\end{equation*}
which can be regarded as analogues of $\Delta_k^\alpha$ and $\delta_k^\alpha$ in Subsection~\ref{sub:partially} in the case of a known error density $f$.
Finally, for $n,m \in \N$, define
\begin{align*}
&\Nalphamn = \inf \{ 1 \leq j \leq n : \alpha_j/(2j+1) <  4d\log(n+4)\omega_j^+/n \} -1 \wedge n,\\
&\Nalphapn = \inf \{ 1 \leq j \leq n : \alpha_j/(2j+1) < \log(n+4)\omega_j^+/(4dn) \} - 1 \wedge n,\\
&\Malphamm = \inf \{ 1 \leq j \leq m : \alpha_j < 4d m^{-1} \log m \} - 1 \wedge m,\\
&\Malphapm = \inf \{ 1 \leq j \leq m : 4d\alpha_j <  m^{-1} \log m \} - 1 \wedge m,
\end{align*}
and set $\Knmalpham = \Nalphamn \wedge \Malphamm$, $\Knmalphap = \Nalphapn \wedge \Malphapm$.
In contrast to the proof of Theorem~\ref{THM:PART:ADAP} we have to impose an additional assumption for the proof of an upper risk bound of $\widehat \lambda_{\kfull}$.

\begin{ass}\label{ass:adap:fully}
	$\exp(-m\alpha_{\Malphapm +1}/(128d)) \leq C(\alpha, d) m^{-5}$ for all $m \in \N$.
\end{ass}

\begin{thm}\label{THM:FULLY:ADAP}
	Let Assumptions~\ref{ass:seq} and~\ref{ass:adap:fully} hold.
	Then, for any $n,m \in \N$,
	\begin{equation*}
	\sup_{\lambda \in \Lambda_\gamma^r} \sup_{f \in \Fc_\alpha^d} \E [ \Vert \widehat \lambda_{\widehat k} - \lambda \Vert_\omega^2 ] \lesssim \min_{0 \leq k \leq \Knmalpham} \max \left\lbrace \frac{\omega_k}{\gamma_k}, \frac{\delta_k^\alpha}{n} \right\rbrace  + \Phi_m + \frac{1}{m} + \frac{1}{n}.
	\end{equation*}
\end{thm}

Note that the only additional prerequisite of Theorem~\ref{THM:FULLY:ADAP} in contrast to Theorem~\ref{THM:PART:ADAP} is the validity of Assumption~\ref{ass:adap:fully}.

\subsection{Examples of convergence rates (continued from Subsection~\ref{subs:minimax:rates})}\label{subs:rates:adap:poisson}

We consider the same configurations for the sequences $\omega$, $\gamma$ and $\alpha$ as in Subsection~\ref{subs:minimax:rates}.
In particular, we assume that $\omega_0=1$ and $\omega_j = \vert j \vert^{2s}$ for all $j \neq 0$.
The different configurations for $\gamma$ and $\alpha$ will be investigated in the following (compare also with the minimax rates of convergence given in Table~\ref{tab:rates}).
Note that the additional Assumption~\ref{ass:adap:fully} is satisfied in all the considered cases.
Let us define $\kndiamond = \argmin_{k\in \Nzero} \max \left\lbrace \omega_k/\gamma_k, \delta_k^\alpha/n \right\rbrace$, that is, $\kndiamond$ realizes the best compromise between squared bias and penalty.

\paragraph{Scenario \emph{\textcolor{red}{(pol)}-\textcolor{blue}{(pol)}:}}
In this scenario, it holds $\kndiamond \asymp n^{1/(2p+2a+1)}$ and $\Nalphamn \asymp ( n / \log n )^{1/(2s+2a+1)}$.
First assume that $\Nalphamn \leq \Malphamm$.
In case that $s < p$, the rate with respect to $n$ is $n^{- 2(p-s)/(2p+2a+1)}$ which is the minimax optimal rate.
In case that $s=p$, it holds $\Nalphamn \lnsim \kndiamond$ and the rate is $(n / \log n)^{-2(p-s)/(2p+2a+1)}$ which is minimax optimal up to a logarithmic factor.
Assume now that $\Malphamm \leq \Nalphamn$. If $\kndiamond \lesssim \Malphamm$, then the estimator obtains the optimal rate with respect to $n$.
Otherwise, $\Malphamm \asymp (m/\log m)^{1/(2a)}$ yields the contribution $(m/\log m)^{-(p-s)/a}$ to the rate. 

\paragraph{Scenario \emph{\textcolor{red}{(exp)}-\textcolor{blue}{(pol)}:}}
$\Nalphamn \asymp (n/\log n)^{1/(2a+2s+1)}$ as in scenario \textcolor{red}{(pol)}-\textcolor{blue}{(pol)}. Since $\kndiamond \asymp \log n$,  it holds $\kndiamond \lesssim \Nalphamn$ and the optimal rate with respect to $n$ holds in case that $\kndiamond \lesssim \Malphamm$.
Otherwise, the bias-penalty tradeoff generates the contribution $(\Malphamm)^{2s} \cdot \exp(-2p \cdot \Malphamm)$ to the rate.

\paragraph{Scenario \emph{\textcolor{red}{(pol)}-\textcolor{blue}{(exp)}:}} It holds that $\kndiamond \asymp \Nalphamn$ and again the sample size $n$ is no obstacle for attaining the optimal rate of convergence.
If $\kndiamond \lesssim \Malphamm$, the optimal rate holds as well.
If $\Malphamm \lnsim \kndiamond$, we get the rate $(\log m)^{-2(p-s)}$ which coincides with the optimal rate with respect to the sample size $m$.

\paragraph{Scenario \emph{\textcolor{red}{(exp)}-\textcolor{blue}{(exp)}:}} We have $\Nalphamn \asymp \log n$
and $k_1 \leq \kndiamond \leq k_2$ where $k_1$ is the solution of $k_1^2\exp((2a+2p)k_1) \asymp n$ and $k_2$ the solution of $\exp((2a+2p)k_2) \asymp n$. Thus, we have $\kndiamond \lnsim \Nalphamn$ and computation of $\omega_{k_1}/\gamma_{k_1}$ resp. $\delta_{k_2}^\alpha/n$ shows that only a loss by a logarithmic factor can occur as far as $\kndiamond \leq \Nalphamn \wedge \Malphamm$. If $\Malphamm \leq \kndiamond$, the contribution to the rate arising from the trade-off between squared bias and penalty is determined by $(\Malphamm)^{2s} \cdot \exp(-2p\Malphamm)$ which deteriorates the optimal rate with respect to $m$ at most by a logarithmic factor.

\appendix

\section{Proofs of Section~\ref{S:MINIMAX}}\label{s:proofs:minimax}

\subsection{Proof of Theorem~\ref{THM:L:N}}

Let us define $\zeta$ as in the statement of the theorem and for each $\theta=(\theta_j)_{0 \leq j \leq k_n^*}\in \{\pm 1\}^{k_n^* + 1}$ the function $\lambda_\theta$ through
\begin{equation*}
\lambda_\theta = \left(\frac r 4 \right)^{1/2} +  \theta_0 \left(\frac{r\zeta}{4n}\right)^{1/2}  + \left(\frac{r\zeta}{4n}\right)^{1/2} \sum_{1 \leq \abs{j} \leq \knast} \theta_\jabs \alpha_j^{-1/2} \e_j\\
\end{equation*}
\noindent Then each $\lambda_\theta$ is a real-valued function by definition which is non-negative since we have
\begin{align*}
\norm{\left(\frac{r\zeta}{4n}\right)^{1/2} \sum_{0 \leq \abs{j} \leq k_n^*} \theta_\jabs \alpha_j^{-1/2} \e_j}_\infty &\leq \left(\frac{r \zeta}{4n}\right)^{1/2} \sum_{0 \leq \abs{j} \leq k_n^*} \alpha_j^{-1/2}\\
&\hspace{-10em}\leq \left(\frac{r\zeta \Gamma}{4} \right)^{1/2} \left( \frac{\gamma_{k_n^*}}{\omega_{k_n^*}} \ \sum_{0 \leq \abs{j} \leq k_n^*} \frac{\omega_j}{n\alpha_j} \right)^{1/2} \leq \left(\frac{r\zeta \eta \Gamma}{4} \right)^{1/2} \leq \left(\frac r 4 \right)^{1/2}.
\end{align*}
Moreover $\norm{\lambda_\theta}_\gamma^2 \leq r$ holds for each $\theta \in \{\pm 1 \}^{\knast+1}$ due to the estimate
\begin{align*}
\norm{\lambda_\theta}_\gamma^2 = \sum_{0 \leq \abs{j} \leq k_n^*} \abs{\fou{\lambda_\theta}_j}^2 \gamma_j
&= \left[ \left(\frac{r}{4}\right)^{1/2}  + \theta_0 \left( \frac{r\zeta}{4n} \right)^{1/2}  \right]^2  + \frac{r\zeta}{4} \, \sum_{1 \leq \abs{j} \leq k_n^*} \frac{\gamma_j}{n\alpha_j}\\
&\leq \frac r 2 + \frac{r\zeta}{2} \frac{\gamma_{k_n^*}}{\omega_{k_n^*}} \sum_{0 \leq \abs{j} \leq k_n^*} \frac{\omega_j}{n\alpha_j} \leq r.
\end{align*}
This estimate and the non-negativity of $\lambda_\theta$ together imply $\lambda_\theta \in \Lambda_\gamma^r$ for all $\theta \in \{ \pm 1\}^{k_n^* + 1}$.
From now on let $f \in \Fc_\alpha^d$ be fixed and let $\P_\theta$ denote the joint distribution of the i.i.d. samples $N_1,\ldots,N_n$ and $Y_1,\ldots,Y_m$ when the true parameters are $\lambda_\theta$ and $f$, respectively.
Let $\P_\theta^{N_i}$ denote the corresponding one-dimensional marginal distributions and $\E_\theta$ the expectation with respect to $\P_\theta$. 
Let $\widetilde \lambda$ be an arbitrary estimator of $\lambda$.
The key argument of the proof is the following reduction scheme:
\begin{align}
\sup_{\lambda \in \Lambda_\gamma^r} \sup_{f \in \Fc_\alpha^d} \E [ \Vert \widetilde \lambda - \lambda \Vert_\omega^2] 
&\geq \frac{1}{2^{k_n^*+1}} \sum \limits_{\theta \in \{\pm 1\}^{k_n^*+1}} \sum_{0 \leq \abs{j} \leq k_n^*} \omega_j \ \E_\theta [\vert \fou{\widetilde \lambda - \lambda_\theta}_j \vert^2]\notag\\
&\hspace{-6em}=\sum_{0 \leq \abs{j} \leq k_n^*} \frac{\omega_j}{2} \sum_{\theta \in \{\pm 1\}^{k_n^*+2}} \{ \E_\theta [\vert \fou{\widetilde \lambda - \lambda_\theta}_j \vert^2] + \E_{\theta^{(j)}} [\vert \fou{\widetilde \lambda - \lambda_{\theta^{(\vert j \vert)}}}_j \vert^2] \},\label{eq:rs:l:n}
\end{align}
where for $\theta \in \{\pm 1\}^{k_n^*+1}$ and $j \in \{ -\knast ,\ldots,\knast \}$ the element $\theta^{(\vert j \vert)} \in \{\pm 1\}^{k_n^*+1}$ is defined by $\theta^{(\vert j \vert)}_k = \theta_k$ for $k \neq \vert j\vert$ and $\theta^{(\vert j \vert)}_{\vert j \vert} = -\theta_{\vert j \vert}$.
Consider the Hellinger affinity $\rho(\P_\theta, \P_{\theta^{(\vert j \vert)}})\defeq \int \sqrt{d\P_{\theta} d\P_{\theta^{(\vert j \vert)}}}$.
For an arbitrary estimator $\widetilde \lambda$ of $\lambda$ we have
\begin{equation*}
\rho(\P_\theta, \P_{\theta^{(\vert j \vert)}}) 
\leq \left(\int \frac{\vert \fou{\widetilde \lambda - \lambda_\theta}_j \vert^2}{\vert \fou{\lambda_\theta - \lambda_{\theta^{(\vert j \vert)}}}_j \vert^2} d\P_\theta \right)^{1/2} + \left(\int \frac{\vert \fou{\widetilde \lambda - \lambda_{\theta^{(\vert j \vert)}}}_j \vert^2}{\vert \fou{\lambda_\theta - \lambda_{\theta^{(\vert j \vert)}}}_j \vert^2} d\P_{\theta^{(\vert j \vert)}} \right)^{1/2}
\end{equation*}
from which we conclude by means of the elementary inequality $(a+b)^2 \leq 2a^2 + 2b^2$ that
\begin{equation*}
\frac 1 2 \vert \fou{\lambda_\theta - \lambda_{\theta^{(\vert j \vert)}}}_j \vert^2 \rho^2(\P_\theta, \P_{\theta^{(\vert j \vert)}}) \leq \E_\theta [\vert \fou{\widetilde \lambda - \lambda_\theta}_j \vert^2] + \E_{\theta^{(\vert j \vert)}} [\vert \fou{\widetilde \lambda - \lambda_{\theta^{(\vert j \vert)}}}_j \vert^2].
\end{equation*}
Define the Hellinger distance between two probability measures $\P$ and $\Q$ as $H(\P,\Q) = (\int [\sqrt{d\P}-\sqrt{d\Q} ]^2 )^{1/2}$ and, analogously, the Hellinger distance between two finite measures $\nu$ and $\mu$ (that not necessarily have total mass equal to one) by $H(\nu, \mu)=(\int [\sqrt{d\nu}-\sqrt{d\mu} ]^2 )^{1/2}$ (as usual, the integral is formed with respect to any measure dominating both $\nu$ and $\mu$).
Let $\nu_\theta$ denote the intensity measure of a Poisson point process $N$ on $[0,1)$ whose Radon-Nikodym derivative with respect to the Lebesgue measure is given by $\ell_\theta = \lambda_\theta \star f$.
Note that we have the estimate $\ell_\theta \geq \delta \sqrt{r}$ for all $\theta \in  \{\pm 1\}^{k_n^*+1}$ with $\delta = \frac{1}{2} - \frac{1}{2\sqrt 2}$ due to
\begin{align*}
\left(\frac{r\zeta}{4n} \right)^{1/2} + \sum_{1 \leq \abs{j} \leq \knast} \abs{\fou{\lambda_\theta}_j \cdot \fou{f}_j} \leq \left(\frac{rd\zeta}{4n}\right)^{1/2} \sum_{0 \leq \abs{j} \leq \knast} \alpha_j^{-1/2} \leq \frac{\sqrt r}{2\sqrt 2} 
\end{align*}
which can be realized in analogy to the non-negativity of $\lambda_\theta$ shown above.
We have
\begin{align*}\label{eq:hell:n}
H^2(\nu_\theta, \nu_{\theta^{(\vert j \vert)}})  = \int \frac{\vert \ell_\theta - \ell_{\theta^{(\vert j \vert)}} \vert^2}{(\sqrt{\ell_\theta} + \sqrt{\ell_{\theta^{(\vert j \vert)}}})^2} \leq \frac{\norm{\ell_\theta-\ell_{\theta^{(\vert j \vert)}}}^2_{2}}{4\delta \sqrt r}
= \frac{\zeta d\sqrt r}{4\delta n}\leq \frac{1}{n}.
\end{align*}
Since the distribution of the sample $Y_1,\ldots,Y_m$ does not depend on the choice of $\theta$ we obtain
\begin{equation*}\label{eq:crux}
H^2(\P_\theta, \P_{\theta^{(\vert j \vert)}}) \leq \sum_{i=1}^n H^2(\P_\theta^{N_i}, \P_{\theta^{(\vert j \vert)}}^{N_i}) \leq \sum_{i=1}^n H^2(\nu_\theta, \nu_{\theta^{(\vert j \vert)}}) \leq 1,
\end{equation*}
where the first estimate follows from Lemma 3.3.10~(i) in~\cite{reiss1989approximate} and the second one is due to Theorem 3.2.1 in~\cite{reiss1993course} which can be applied since each $N_i$ is a Poisson point process for the Poisson model.
Thus, the relation $\rho(\P_\theta, \P_{\theta^{(\vert j \vert)}})=1-\frac 1 2 H^2(\P_\theta, \P_{\theta^{(\vert j \vert)}})$ implies $\rho(\P_\theta, \P_{\theta^{(\vert j \vert)}}) \geq \frac 1 2$.
Finally, putting the obtained estimates into the reduction scheme~\eqref{eq:rs:l:n} leads to
\begin{align*}
\sup_{\lambda \in \Lambda_\gamma^r} \sup_{f \in \Fc_\alpha^d} \E [ \Vert \widetilde \lambda - \lambda \Vert_\omega^2] 
&\geq \sum_{0 \leq \abs{j} \leq \knast} \frac{\omega_j}{16} \vert \fou{\lambda_\theta - \lambda_{\theta^{(\vert j \vert)}}}_j \vert^2
\geq \frac{\zeta r}{16\eta} \cdot \Psi_n
\end{align*}
which finishes the proof of the theorem since $\widetilde \lambda$ was arbitrary. \hfill $\qedsymbol$

\subsection{Proof of Theorem~\ref{THM:L:M}}

By Markov's inequality we have for an arbitrary estimator $\widetilde \lambda$ of $\lambda$ and $A > 0$ (which will be specified below)
\begin{equation*}
\E [  \Phi_m^{-1} \Vert \widetilde \lambda - \lambda\Vert_\omega^2 ] \geq A \cdot \P (\Vert \widetilde \lambda - \lambda\Vert_\omega^2 \geq A \Phi_m ),
\end{equation*}
which by reduction to two hypotheses implies
\begin{align*}
\sup_{\lambda \in \Lambda_\gamma^r}  \sup_{f \in \Fc_\alpha^d} \E [ \Phi_m^{-1}\Vert \widetilde \lambda - \lambda\Vert_\omega^2 ]  
&\geq A \cdot \sup_{\theta \in \{\pm 1\}} \P_{\theta} (\Vert \widetilde \lambda - \lambda_\theta\Vert_\omega^2 \geq A \Phi_m )
\end{align*}
where $\P_\theta$ denotes the distribution when the true parameters are $\lambda_\theta$ and $f_\theta$. The specific hypotheses $\lambda_1, \lambda_{-1}$ and $f_1, f_{-1}$ will be specified below.
If $\lambda_{-1}$ and $\lambda_1$ can be chosen such that $\norm{\lambda_{1}-\lambda_{-1}}_\omega^2 \geq 4 A \Phi_m$, application of the triangle inequality yields
\begin{equation*}
\P_\theta (  \Vert \widetilde \lambda - \lambda_\theta\Vert_\omega^2 \geq A \Phi_m )  \geq \P_\theta (\tau^* \neq \theta )
\end{equation*}
where $\tau^*$ is the \emph{minimum distance test} given by $\tau^* = \arg \min_{\theta \in \{\pm 1\}} \Vert\widetilde \lambda - \lambda_\theta\Vert_\omega^2$.
Hence, we obtain
\begin{align}
\inf_{\widetilde \lambda} \sup_{\lambda \in \Lambda_\gamma^r}  \sup_{f \in \Fc_\alpha^d} \P ( \Vert \widetilde \lambda - \lambda\Vert_\omega^2 \geq A \Phi_m ) &\geq \inf_{\widetilde \lambda} \sup_{\theta \in \{\pm 1\}} \P_\theta (\Vert\widetilde \lambda - \lambda_\theta\Vert_\omega^2 \geq A \Phi_m ) \nonumber \\
&\geq \inf_{\tau} \sup_{\theta \in \{\pm 1\}} \P_\theta \left(\tau \neq \theta \right) \eqdef p^*\label{eq:df:pstar}
\end{align}
where the infimum is taken over all $\{ \pm 1 \}$-valued functions $\tau$ based on the observations.
Thus, it remains to find hypotheses $\lambda_1, \lambda_{-1} \in \Lambda_\gamma^r$ and $f_1, f_{-1} \in \Fc_\alpha^d$ such that
\begin{equation}\label{eq:A:m}
\norm{\lambda_1 - \lambda_{-1}}_\omega^2 \geq 4 A \Phi_m,
\end{equation}
and which allow us to bound $p^*$ by a universal constant (independent of $m$) from below.
For this purpose, set $\kmast=\arg \max_{j \geq 1} \{ \omega_j/\gamma_j \min (1,\frac{1}{m\alpha_j} ) \}$ and $a_m=\zeta \min (1, m^{-1/2}\alpha_{k_m^*}^{-1/2} )$, where $\zeta$ is defined as in the statement of the theorem.
Take note of the inequalities
$1/d^{1/2} = (1-(1-1/d^{1/4}) )^2 \leq (1-a_m)^2 \leq 1$
and
$1 \leq (1+a_m)^2 \leq ( 1 + (1-1/d^{1/4}))^2 = (2-1/d^{1/4})^2 \leq d^{1/2} $
which in combination imply $1/d^{1/2} \leq (1+\theta a_m)^2 \leq d^{1/2}$ for $\theta \in \{\pm 1\}$.
These inequalities will be used below without further reference.
For $\theta \in \{\pm 1\}$, we define
\begin{equation*}
\lambda_\theta= \left( \frac{r}{2} \right)^{1/2}  + (1-\theta a_m) \ \left( \frac{r}{8} \right)^{1/2} d^{-1/4} \ \gamma_{k_m^*}^{-1/2}  ( \e_{k_m^*} +   \e_{-\kmast} ).
\end{equation*}
Furthermore, we have
$$\norm{\lambda_\theta}_\gamma^2= \frac{r}{2} + 2\gamma_{k_m^*} \vert\fou{\lambda_\theta}_{k_m^*}\vert^2 \leq \frac{r}{2} + (1 + a_m)^2 \frac{r}{4}d^{-1/2} \leq \frac{3r}{4}$$ and
$\abs{\lambda_\theta(t)} \geq \left( \frac{r}{2} \right)^{1/2}  - 2\left( \frac{r}{8} \right)^{1/2} \geq 0 \text{ for all } t \in [0,1)$
which together imply that $\lambda_\theta \in \Lambda_\gamma^r$ for $\theta \in \{\pm 1\}$.
The identity 
$$\norm{\lambda_1-\lambda_{-1}}_\omega^2=r a_m^2d^{-1/2}\omega_{k_m^*}\gamma_{k_m^*}^{-1}= \zeta^2 r d^{-1/2} \cdot \Phi_m$$
shows that the condition in~\eqref{eq:A:m} is satisfied with $A=\zeta^2 r/(4\sqrt d)$.

Let $f \in \Fc_\alpha^{\sqrt d}$ be such that $f \geq 1/2$ (the existence is guaranteed through condition (\textcolor{black}{C4})) and define for $\theta \in \{\pm 1\}$
\begin{equation*}
f_\theta = f + \theta a_m (\fou{f}_{k_m^*} \e_{k_m^*} + \fou{f}_{-k_m^*} \e_{-k_m^*}).
\end{equation*}
Since $k_m^* \geq 1$ we have $\int_0^1 f_\theta(x)dx = 1$ and $f_\theta \geq 0$ holds because of the estimate 
$$\abs{f_\theta(t)} \geq 1/2 - 2a_m \alpha_{k_m^*}^{1/2} d^{1/2} \geq 0 \qquad \text{for all } t.$$
For $\jabs \neq k_m^*$,  we have $\fou{f}_j = \fou{f_\theta}_j$ and thus trivially $1/d \leq \vert \fou{f_\theta}_j \vert^2/\alpha_j \leq d$ for $\jabs \neq k_m^*$ since $\Fc_\alpha^{\sqrt d} \subset \Fc_\alpha^{d}$.
Moreover
\begin{equation*}
1/d \leq d^{-1/2} \frac{\vert \fou{f}_{\pm k_m^*} \vert^2}{\alpha_{\pm k_m^*}} \leq \frac{(1+\theta a_m)^2 \vert \fou{f}_{\pm k_m^*} \vert^2}{\alpha_{\pm k_m^*}} \leq d^{1/2} \frac{\vert \fou{f}_{\pm k_m^*} \vert^2}{\alpha_{\pm k_m^*}} \leq d
\end{equation*}
and hence $f_\theta \in \Fc_\alpha^d$ for $\theta \in \{\pm 1\}$.

To obtain a lower bound for $p^\ast$ defined in~\eqref{eq:df:pstar} consider the joint distribution $\P_\theta$ of the samples $N^1,\ldots,N^n$ and $Y_1,\ldots,Y_m$ under $\lambda_\theta$ and $f_\theta$.
Note that due to our construction we have 
$\lambda_{-1} \star f_{-1}=\lambda_1 \star f_1$. Thus $\P_{-1}^{N^i} = \P_{1}^{N^i}$
for all $i=1,\ldots,n$ (due to the fact that the distribution of a \emph{Poisson} point process is determined by its intensity) and the Hellinger distance between $\P_{-1}$ and $\P_1$ does only depend on the distribution of the sample $Y_1,\ldots,Y_m$.
More precisely,
\begin{equation*}
H^2(\P_{-1},\P_{1}) = H^2(\P_{-1}^{Y_1,\ldots,Y_m},\P_{1}^{Y_1,\ldots,Y_m}) \leq m H^2(\P_{-1}^{Y_1},\P_{1}^{Y_1}),
\end{equation*}
and we proceed by bounding $H^2(\P_{-1}^{Y_1},\P_{1}^{Y_1})$ from above.
Recall that $f \geq 1/2$ which is used to obtain the estimate
\begin{align*}
H^2(\P_{-1}^{Y_1},\P_{1}^{Y_1}) = \int_0^1 \frac{\abs{f_1(x)-f_{-1}(x)}^2}{2f(x)}dx&\leq  \int \abs{f_1(x)-f_{-1}(x)}^2 dx \leq \frac{1}{m}.
\end{align*}
Hence we have $H^2(\P_{-1},\P_{1}) \leq 1$ and application of statement (ii) of Theorem 2.2 in~\cite{tsybakov2009introduction} with $\alpha=1$ implies $p^\ast \geq \frac{1}{2} (1-\sqrt{3}/2)$ which finishes the proof of the theorem.  \hfill $\qedsymbol$

\subsection{Proof of Theorem~\ref{THM:U:NM:I:II}}

Set $\widetilde \lambda_{k_n^\ast}= \sum_{0 \leq \abs{j} \leq k_n^\ast} \fou{\lambda}_j \1_{\Omega_j} \e_j$.
The proof consists in finding appropriate upper bounds for the quantities $\square$ and $\triangle$ in the estimate
\begin{equation*}\label{eq:AB}
\E [\Vert \widehat \lambda_{k_n^\ast}-\lambda\Vert_\omega^2] \leq 2 \, \E [\Vert\widehat \lambda_{k_n^\ast}-\widetilde \lambda_{k_n^\ast}\Vert_\omega^2] + 2 \, \E [\Vert\lambda - \widetilde \lambda_{k_n^\ast}\Vert_\omega^2] \eqdef 2 \square + 2 \triangle.
\end{equation*}

\noindent \emph{Upper bound for $\square$}: Using the identity $\E \widehat{\fou{\ell}}_j=\fou{f}_j \fou{\lambda}_j$ we obtain
\begin{align*}
\square &= \sum_{0 \leq \abs{j} \leq k_n^\ast} \omega_j \, \E [ \vert \widehat{\fou{\ell}}_j/\widehat{\fou{f}}_j- \fou{\lambda}_j\vert^2 \1_{\Omega_j} ] \\
&\leq 2 \sum_{0 \leq \abs j \leq k_n^\ast} \omega_j \, \E [  \vert \widehat{\fou{\ell}}_j/\widehat{\fou{f}}_j- \E  \widehat{\fou{\ell}}_j /\widehat{\fou{f}}_j\vert^2 \, \1_{\Omega_j} ]\\
&\hspace{1em}+ 2 \sum_{0 \leq \abs j \leq k_n^\ast} \omega_j \vert \fou{\lambda}_j\vert^2 \, \E [  \vert \fou{f}_j/\widehat{\fou{f}}_j - 1\vert^2 \, \1_{\Omega_j} ] \eqdef 2 \square_1 + 2 \square_2.
\end{align*}
Using the estimate $\abs{a}^2 \leq 2\abs{a-1}^2 + 2$ for $a=\fou{f}_j/\widehat{\fou{f}}_j$, the definition of $\Omega_j$ and the independence of $\widehat{\fou{\ell}}_j$ and $\widehat{\fou{f}}_j$ we get
\begin{align*}
\square_1 &= \sum_{0 \leq \abs j \leq k_n^\ast} \omega_j \, \E\left[  \vert\widehat{\fou{\ell}}_j/\widehat{\fou{f}}_j- \E   \widehat{\fou{\ell}}_j    /\widehat{\fou{f}}_j\vert^2 \cdot \abs{\frac{\fou{f}_j}{\fou{f}_j}}^2 \1_{\Omega_j}\right]\\
&\leq 2 \sum_{0 \leq \abs j \leq k_n^\ast} m\omega_j \frac{\var (\widehat{\fou{\ell}}_j)  \var ( \widehat{\fou{f}}_j ) }{\abs{\fou{f}_j}^2} + 2 \sum_{0 \leq \abs j \leq k_n^\ast} \omega_j \frac{\var( \widehat{\fou{\ell}}_j )}{\abs{\fou{f}_j}^2}.
\end{align*}
Applying statements~\ref{it:l:u:a} and~\ref{it:l:u:b} from Lemma~\ref{l:upper} together with $f \in \Fc_\alpha^d$ yields
\begin{equation*}
\square_1 \leq 4d \sum_{0 \leq \abs j \leq k_n^\ast} \omega_j \, \frac{ \fou{\lambda}_0}{n\alpha_j}
\end{equation*}
which using that $\gamma_j \geq 1$ (which holds due to Assumption~\ref{ass:seq}) implies
\begin{equation*}
\square_1 \leq 4d\sqrt r \sum_{0 \leq \abs j \leq k_n^\ast} \frac{\omega_j}{n\alpha_j} \leq 4d\sqrt r \cdot \Psi_n.
\end{equation*}
Now consider $\square_2$.
Using the estimate $\abs{a}^2 \leq 2\abs{a-1}^2 + 2$ for $a=\fou{f}_j/\widehat{\fou{f}}_j$ and the definition of $\Omega_j$ yields
\begin{equation}\label{eq:A2:1}
\E [  \vert \fou{f}_j/\widehat{\fou{f}}_j - 1\vert^2 \, \1_{\Omega_j} ] \leq 2m \, \frac{\E  [\vert\widehat{\fou{f}}_j - \fou{f}_j\vert^4] }{\abs{\fou{f}_j}^2} + 2 \, \frac{\var ( \widehat{\fou{f}}_j ) }{\abs{\fou{f}_j}^2}.
\end{equation}
Notice that Theorem~2.10 in~\cite{petrov1995limit} implies the existence of a constant $C > 0$ with $\E [\vert\widehat{\fou{f}}_j-\fou{f}_j\vert^4] \leq C/m^2$.
Using this inequality in combination with assertion~\ref{it:l:u:b} from Lemma~\ref{l:upper} and $f \in \Fc_\alpha^d$ implies
\begin{equation}\label{eq:A2:1}
\E [  \vert \fou{f}_j/\widehat{\fou{f}}_j - 1\vert^2 \, \1_{\Omega_j} ] \leq 2d (C + 1)/(m\alpha_j).
\end{equation}
In addition, $\E [  \vert \fou{f}_j/\widehat{\fou{f}}_j - 1\vert^2 \1_{\Omega_j} ] \leq m \var( \widehat{\fou{f}}_j) \leq 1$ which in combination with~\eqref{eq:A2:1} implies
\begin{equation*}
\square_2 \leq 2d(C+1) \sum_{0 \leq \abs j \leq k_n^\ast}\omega_j \abs{\fou{\lambda}_j}^2 \min \{ 1,  1/(m\alpha_j) \}. 
\end{equation*}
Exploiting the fact that $\lambda \in \Lambda_\gamma^r$ and the definition of $\Phi_m$ in~\eqref{eq:def:Phi:m} we obtain 
$$\square_2 \leq 2dr(C+1)(1+\gamma_1/\omega_1)\cdot \Phi_m.$$ 
Putting together the estimates for $\square_1$ and $\square_2$ yields 
$$\square \leq 8d\sqrt r \cdot \Psi_n + 4d(C+1)(1+\gamma_1/\omega_1)r \cdot \Phi_m.$$
\noindent \emph{Upper bound for $\triangle$}: $\triangle$ can be decomposed as
\begin{align*}
\triangle &= \sum_{j \in \Z} \omega_j \abs{\fou{\lambda}_j}^2 \, \E (1-\1_{\{0 \leq \abs j \leq k_n^*\}} \cdot \1_{\Omega_j} )\\
&=\sum_{\abs j > k_n^*} \omega_j \abs{\fou{\lambda}_j}^2  + \sum_{0 \leq \abs j \leq k_n^*} \omega_j \abs{\fou{\lambda}_j}^2 \cdot \P( \Omega_j^\complement ) = \triangle_1 + \triangle_2.
\end{align*}
$\lambda \in \Lambda_\gamma^r$ implies $\triangle_1 \leq r \omega_{k_n^*}/\gamma_{k_n^*}\leq r \cdot \Psi_n$ and Lemma~\ref{l:upper} yields the estimate $\triangle_2 \leq 4dr \cdot \Phi_m$ which together imply $\triangle \leq r \cdot \Psi_n + 4dr \cdot \Phi_m$.
Combining the derived estimates for $\square$ and $\triangle$ finishes the proof.\qed

\subsection{Auxiliary results for the proof of Theorem~\ref{THM:U:NM:I:II}}\label{s:aux}
\begin{lem}\label{l:upper} The following assertions hold:
	\begin{enumerate}[label=\alph*)]
		\item\label{it:l:u:a} $\var ( \widehat{\fou{\ell}}_j ) \leq \fou{\lambda}_0/n$,
		
		\item\label{it:l:u:b} $\var ( \widehat{\fou{f}}_j)  \leq 1/m$,
		\item\label{it:l:u:c} $\P (  \Omega_j^\complement )   = \P (\vert \widehat{\fou{f}}_j\vert^2 < 1/m )   \leq \min \left\lbrace 1, 4d/(m\alpha_j) \right\rbrace \quad \forall f \in \Fc_\alpha^d$.
	\end{enumerate}
\end{lem}

\begin{proof}
	The proof of statement~\ref{it:l:u:a} is given by the identity
	\begin{equation*}
	\var( \widehat{\fou{\ell}}_j )  = \frac 1 n \, \var \left(  \int_0^1 \e_j(t)dN_1(t) \right)   = \frac 1 n \, \int_0^1 \vert \e_j(t)\vert^2 (\lambda \star f)(t)dt = \frac 1 n \cdot \fou{\lambda}_0.
	\end{equation*}
	
	For the proof of~\ref{it:l:u:b}, note that we have $\var( \widehat{\fou{f}}_j) =\frac{1}{m} \, \var\left( \e_j(-Y_1)\right)$ and the assertion follows from the estimate
	$$\var ( \e_j(-Y_1) ) = \E [ \abs{\e_j(-Y_1)}^2 ] - \abs{\E \left[ \e_j(-Y_1) \right] }^2 \leq \E [ \abs{\e_j(-Y_1)}^2 ] = 1.$$
	
	For the proof of~\ref{it:l:u:c}, we consider two cases: if $\abs{\fou{f}_j}^2 < 4/m$ we have $1 < \frac{4d}{m\alpha_j}$ because $f \in \Fc_\alpha^d$ and the statement is evident.
	Otherwise, $\abs{\fou{f}_j}^2 \geq 4/m$ which implies
	\begin{align*}
	\P ( \vert \widehat{\fou{f}}_j \vert^2 < 1/m  ) &\leq  \P ( \vert \widehat{\fou{f}}_j\vert/\abs{\fou{f}_j} < 1/2 )\leq \P ( \vert \widehat{\fou{f}}_j/\fou{f}_j - 1\vert > 1/2 ).
	\end{align*}
	Applying Chebyshev's inequality and exploiting the definition of $\Fc_\alpha^d$ yields
	\begin{equation*}
	\P ( \vert \widehat{\fou{f}}_j \vert^2 < 1/m  ) \leq 4/\abs{\fou{f}_j}^2 \cdot \var( \widehat{\fou{f}}_j)  \leq 4d/(m\alpha_j)
	\end{equation*}
	and statement~\ref{it:l:u:c} follows.
\end{proof}
\section{Proofs of Section~\ref{S:ADAP}}\label{s:proofs:adap:poisson}

\subsection{Proof of Theorem~\ref{THM:PART:ADAP}}

Define the events $\Xi_1 = \{ (\fou{\ell}_0 \vee 1)/2 \leq \widehat{\fou{\ell}}_0 \vee 1 \leq 2 (\fou{\ell}_0 \vee 1) \}$ and
\begin{equation*}
\Xi_2 = \left\lbrace \forall \, 0 \leq \vert j \vert \leq M_m^\alpha : \vert \widehat{\fou{f}}_j^{-1} - \fou{f}_j^{-1} \vert \leq \frac{1}{2 \vert \fou{f}_j \vert} \quad \text{and} \quad \vert \widehat{\fou{f}}_j \vert \geq \frac{1}{m}  \right\rbrace.
\end{equation*}
The identity $1=\1_{\Xi_1 \cap \Xi_2} + \1_{\Xi_2^\complement}+ \1_{\Xi_1^\complement \cap \Xi_2}$ provides the decomposition
\begin{equation*}
\E \Vert \widehat \lambda_\kpart - \lambda \Vert_\omega^2 = \underbrace{\E [ \Vert \widehat \lambda_\kpart - \lambda \Vert_\omega^2 \1_{\Xi_1 \cap \Xi_2} ]}_{\eqdef \square_1} + \underbrace{\E [  \Vert \widehat \lambda_\kpart - \lambda \Vert_\omega^2 \1_{\Xi_2^\complement} ]}_{\eqdef \square_2} + \underbrace{\E [  \Vert \widehat \lambda_\kpart - \lambda \Vert_\omega^2 \1_{\Xi_1^\complement \cap \Xi_2} ]}_{\eqdef \square_3}, 
\end{equation*}
and we will establish uniform upper bounds over the ellipsoids $\Lambda_\gamma^r$ and $\Fc_\alpha^d$ for the three terms on the right-hand side separately.

\noindent \emph{Uniform upper bound for $\square_1$}:
Denote by $\Sc_k$ the linear subspace of $\L^2$ spanned by the functions $\e_j(\cdot)$ for $j \in \{ -k,\ldots,k \}$.
Since the identity $\Upsilon(t) = \Vert t-\widehat \lambda_k\Vert_\omega^2 - \Vert \widehat \lambda_k\Vert_\omega^2$ holds for all $t \in \Sc_k$, $k \in \{0,\ldots,n \wedge m\}$, we obtain for all such $k$ that $\argmin_{t \in \Sc_k} \Upsilon_{}(t) = \widehat \lambda_k$.
Using this identity and the definition of $\kpart$ yields for all $k \in \{0,\ldots,\Knmalpha \}$ that
\begin{equation*}
\Upsilon(\widehat \lambda_{\widetilde k}) + \tildepen_{\widetilde k} \leq \Upsilon(\widehat \lambda_k) + \tildepen_k \leq \Upsilon(\lambda_k) + \tildepen_k
\end{equation*}
where $\lambda_k = \sum_{0 \leq \vert j \vert \leq k} \fou{\lambda}_j \e_j$ denotes the projection of $\lambda$ on the subspace $\Sc_k$.
Elementary computations imply
\begin{equation}\label{eq:part:0}
\Vert \widehat \lambda_{\widetilde k}\Vert_\omega^2 \leq \Vert \lambda_k\Vert_\omega^2 + 2 \Re\langle \widehat \lambda_{n \wedge m},\widehat \lambda_{\widetilde k}- \lambda_k \rangle_\omega + \tildepen_k - \tildepen_{\widetilde k}
\end{equation}
for all $k \in \{0,\ldots, \Knmalpha \}$.
In addition to $\lambda_k$ defined above, introduce the further abbreviations
\begin{align*}
\widetilde \lambda_k = \sum_{0 \leq \vert j \vert \leq k} \frac{\widehat{\fou{\ell}}_j}{\fou{f}_j} \e_j \qquad \text{and} \qquad \check{\lambda}_k = \sum_{0 \leq \vert j \vert \leq k} \frac{\fou{\ell}_j}{\widehat{\fou{f}}_j} \1_{\Omega_j} \e_j,
\end{align*}
as well as $\Theta_{k} = \widehat \lambda_{k} - \check \lambda_k -\widetilde \lambda_k + \lambda_k, \widetilde \Theta_{k} = \widetilde \lambda_{k} - \lambda_k, \text{ and }  \check \Theta_k = \check \lambda_k - \lambda_k$.
Using these abbrevations and the identity $\widehat \lambda_{n \wedge m} - \lambda_{n \wedge m}= \Theta_{n \wedge m} + \widetilde \Theta_{n \wedge m} + \check \Theta_{n \wedge m}$, we deduce from~\eqref{eq:part:0} that
\begin{align}\label{eq:part:1}
\Vert \widehat \lambda_{\widetilde k} - \lambda\Vert_\omega^2 &\leq \norm{\lambda - \lambda_k}_\omega^2 + \tildepen_k - \tildepen_{\widetilde k}
+ 2 \Re \langle  \widetilde \Theta_{n \wedge m},\widehat \lambda_{\widetilde k}-\lambda_k \rangle_\omega \notag\\
&\hspace{1em}+ 2 \Re \langle \Theta_{n \wedge m},\widehat \lambda_{\widetilde k} - \lambda_k\rangle_\omega + 2 \Re \langle \check \Theta_{n \wedge m},\widehat \lambda_{\widetilde k} - \lambda_k\rangle_\omega
\end{align}
for all $k \in \{0,\ldots, \Knmalpha \}$.
Define $\mathcal B_k = \{\lambda \in \mathcal S_k : \norm{\lambda}_\omega \leq 1\}$. 
For every $\tau > 0$ and $t \in \Sc_k$, the estimate $2uv \leq \tau u^2 + \tau^{-1}v^2$ implies
\begin{equation*}
2 \abs{\langle h,t \rangle_\omega} \leq 2 \norm{t}_\omega \sup_{t \in \Bc_k} \abs{\langle h,t \rangle_\omega} \leq \tau \norm{t}_\omega^2 + \frac 1 \tau \sup_{t \in \Bc_k} \abs{\langle h,t \rangle_\omega}^2.
\end{equation*}
Because $\widehat \lambda_{\widetilde k}-\lambda_k \in \Sc_{\widetilde k \vee k}$, combining the last estimate with~\eqref{eq:part:1} we get
\begin{align*}
\Vert \widehat \lambda_{\widetilde k}-\lambda\Vert_\omega^2 &\leq \norm{\lambda- \lambda_k}_\omega^2 + 3 \tau \Vert\widehat \lambda_{\widetilde k} - \lambda_k\Vert_\omega^2 + \tildepen_k - \tildepen_{\widetilde k} +\\
&\hspace{-4em}+ \tau^{-1} ( \sup_{t \in \Bc_{k \vee \widetilde k}} \vert \langle \widetilde \Theta_{n \wedge m},t \rangle_\omega\vert ^2 + \sup_{t \in \Bc_{k \vee \widetilde k}} \vert \langle \Theta_{n \wedge m},t  \rangle_\omega\vert^2
+ \sup_{t \in \Bc_{k \vee \widetilde k}} \vert\langle \check \Theta_{n\wedge m},t \rangle_\omega\vert^2).
\end{align*}
Note that $\Vert \widehat \lambda_{\widetilde k} - \lambda_k\Vert_\omega^2 \leq 2 \Vert \widehat \lambda_{\widetilde k} - \lambda\Vert_\omega^2 + 2 \norm{\lambda_k - \lambda}_\omega^2$ and $\norm{\lambda - \lambda_k}_\omega^2 \leq r\omega_k\gamma_k^{-1}$ for all $\lambda \in \Lambda_\gamma^r$ since $\omega\gamma^{-1}$ is non-increasing due to Assumption~\ref{ass:seq}.
Specializing with $\tau = 1/8$, we obtain
\begin{align}\label{eq:tau:1/8}
\Vert \widehat \lambda_{\widetilde k} - \lambda\Vert_\omega^2 &\leq 7 r\omega_k\gamma_k^{-1} + 4\tildepen_k - 4\tildepen_{\widetilde k} + 32 \sup_{t \in \Bc_{k \vee \widetilde k}} \vert\langle \widetilde \Theta_{n \wedge m},t \rangle_\omega\vert^2\notag\\
&\hspace{1em}+ 32\sup_{t \in \Bc_{k \vee \widetilde k}} \vert\langle \Theta_{n\wedge m},t \rangle_\omega\vert^2+ 32 \sup_{t \in \Bc_{k \vee \widetilde k}} \vert\langle \check \Theta_{n \wedge m},t \rangle_\omega\vert^2.
\end{align}
Combining the facts that $\1_{\Omega_j} \1_{\Xi_2} = \1_{\Xi_2}$ for $0 \leq \vert j \vert \leq M_m^\alpha$ and $\Knmalpha \leq M_m^\alpha$ by definition, we obtain for all $j \in \{ -\Knmalpha,\ldots,\Knmalpha \}$ the estimate
\begin{equation*}
\vert\fou{f}_j/\widehat{\fou{f}}_j \1_{\Omega_j} -1 \vert^2 \1_{\Xi_2} = \abs{\fou{f}_j}^2 \vert 1/\widehat{\fou{f}}_j - 1/\fou{f}_j\vert^2 \1_{\Xi_2} \leq 1/4.
\end{equation*}
Hence, $\sup_{t \in \Bc_k} \vert\langle \Theta_{n \wedge m}, t \rangle_\omega\vert^2 \1_{\Xi_2} \leq \frac 1 4 \sup_{t \in \Bc_k} \vert\langle \widetilde \Theta_{n \wedge m},t \rangle_\omega\vert^2$ for all $0 \leq k \leq \Knmalpha$.
Thus, from~\eqref{eq:tau:1/8} we obtain
\begin{equation*}\label{eq:part:2}
\begin{split}
\Vert\widehat \lambda_{\widetilde k} - \lambda\Vert_\omega^2 \, \1_{\Xi_1 \cap \Xi_2} &\leq 7r\frac{\omega_k}{\gamma_k} + 40 \left(\sup_{t \in \Bc_{k \vee \kpart}} \vert\langle \widetilde \Theta_{n \wedge m},t \rangle_\omega\vert^2 - \frac{33d(\fou{\ell}_0 \vee 1) \delta_{k \vee \kpart}^\alpha}{8n} \right)_+\\
&\hspace{-8em}+(165d(\fou{\ell}_0 \vee 1) \delta_{k \vee \widetilde k}^\alpha/n + 4\tildepen_k - 4\tildepen_{\widetilde k})\1_{\Xi_1 \cap \Xi_2} + 32 \sup_{t \in \Bc_{K_{nm}^\alpha}} \vert\langle \check \Theta_{n \wedge m},t \rangle_\omega\vert^2.
\end{split}
\end{equation*}
Exploiting the definition of both the penalty $\tildepen$ and the event $\Xi_1$, we obtain
\begin{align}\label{eq:term:mho1}
\E [ \Vert \widehat \lambda_{\widetilde k} - \lambda\Vert_\omega^2 \,  \1_{\Xi_1 \cap \Xi_2}] &\leq C(d,r) \min_{0 \leq k \leq K_{nm}^\alpha} \max \left\lbrace \frac{\omega_k}{\gamma_k}, \frac{\delta_k^\alpha}{n} \right\rbrace \notag\\
&\hspace{+0em}+40 \sum_{k=0}^{K_{nm}^\alpha} \E \left[ \left( \sup_{t \in \Bc_{k}} \vert\langle \widetilde \Theta_{n \wedge m},t \rangle_\omega\vert^2 - \frac{33(\fou{\ell}_0 \vee 1)d\delta_{k}^\alpha}{8n} \right)_+\right] \notag\\
&\hspace{+0em}+ 32 \E \left[  \sup_{t \in \Bc_\Knmalpha} \vert\langle \check \Theta_{n \wedge m},t \rangle_\omega\vert^2\right].
\end{align}
Applying Lemma~\ref{l:ex:conc} with $\delta_k^*=d\delta_k^\alpha$ and $\Delta_k^*=d\Delta_k^\alpha$ yields
\begin{align*}
&\E \left[ \left( \sup_{t \in \Bc_{k}} \vert \langle \widetilde \Theta_{n \wedge m},t \rangle_\omega\vert^2 - \frac{33d(\fou{\ell}_0 \vee 1) \delta_k^\alpha}{8n} \right)_+\right]\\
&\hspace{1em}\leq K_1 \left[ \frac{d\Vert f \Vert \Vert \lambda \Vert \Delta_k^\alpha}{n} \exp \left( - K_2 \frac{\delta_k^\alpha}{\Vert f \Vert^2 \Vert \lambda \Vert^2 \Delta_k^\alpha} \right) + \quad  \frac{d\delta_k^\alpha}{n^2} \exp(-K_3 \sqrt n) \right].  
\end{align*}
Using Statement~\ref{it:l:B2:a} of Lemma~\ref{l:B2} and the fact that $K_{nm}^\alpha \leq n$ by definition, we obtain that
\begin{align*}
\sum_{k=0}^{K_{nm}^\alpha} \E \left[ \left(\sup_{t \in \Bc_k} \vert \langle \widetilde \Theta_{n \wedge m},t \rangle_\omega\vert - \frac{33d(\fou{\ell}_0 \vee 1) \delta_k^\alpha}{8n} \right)_+\right]&\\
&\hspace{-19em}\lesssim \frac{d^{3/2} \sqrt{r\rho} }{n} \sum_{k=0}^{\infty} \Delta_k^\alpha \exp \left( - \frac{2K_2 k}{ \sqrt{dr\rho} } \cdot \frac{\log(\Delta_k^\alpha \vee (k+3))}{ \log (k+3)} \right) + \exp(-K_3 \sqrt n),
\end{align*}
where the last estimate is due to the fact that $\norm{f}^2 \leq d \rho$ for all $f \in \Fc_\alpha^d$ and $\norm{\lambda}^2 \leq r$ for all $\lambda \in \Lambda_\gamma^r$.
Note that we have $$\sum_{k=0}^{\infty} \Delta_k^\alpha \exp \left( - \frac{2K_2 k}{ \sqrt{dr\rho} } \cdot \frac{\log(\Delta_k^\alpha \vee (k+3))}{ \log (k+3)} \right) \leq C < \infty$$ with a numerical constant $C$ which implies
\begin{equation*}
\sum_{k=0}^{K_{nm}^\alpha} \E \left[ \left(\sup_{t \in \Bc_k} \vert \langle \widetilde \Theta_{n \wedge m},t \rangle_\omega\vert - \frac{33d(\fou{\ell}_0 \vee 1) \delta_k^\alpha}{8n} \right)_+\right] \lesssim \frac{1}{n}.
\end{equation*}
The last term in~\eqref{eq:term:mho1} is bounded by means of Lemma~\ref{l:B4} which immediately yields
$\E [ \sup_{t \in \Bc_{K_{nm}^\alpha}} \vert\langle \check \Theta_{n \wedge m}, t \rangle_\omega \vert^2 ]  \lesssim \Phi_m$.
Combining the preceding estimates, which hold uniformly for all $\lambda \in \Lambda_\gamma^r$ and $f \in \Fc_\alpha^d$, we conclude from Equation~\eqref{eq:term:mho1} that
\begin{align*}
\sup_{\lambda \in \Lambda_\gamma^r} \sup_{f \in \Fc_\alpha^d} \E \left[ \Vert\widehat \lambda_{\widetilde k} - \lambda\Vert_\omega^2 \, \1_{\Xi_1 \cap \Xi_2}\right] \lesssim \min_{0 \leq k \leq K_{nm}^\alpha} \max \left\lbrace \frac{\omega_k}{\gamma_k}, \frac{\delta_k^\alpha}{n} \right\rbrace  + \Phi_m + \frac{1}{n}.
\end{align*}

\noindent \emph{Uniform upper bound for $\square_2$}:
Define $\breve \lambda_k = \sum_{0 \leq \abs j \leq k} \fou{\lambda}_j \1_{\Omega_j} \e_j$. Note that $\Vert\widehat \lambda_k - \breve \lambda_k\Vert_\omega^2 \leq \Vert\widehat \lambda_{k'} - \breve \lambda_{k'}\Vert_\omega^2$ for $k \leq k'$ and $\Vert\breve \lambda_k - \lambda\Vert_\omega^2 \leq \Vert\lambda\Vert_\omega^2$ for all $k \in \N_0$.
Consequently, since $k \in \{ 0,\ldots,\Knmalpha \}$, we obtain the estimate
\begin{align*}
\E [ \Vert\widehat \lambda_{\widetilde k} - \lambda\Vert_\omega^2 \, \1_{\Xi_2^\complement}]  &\leq 2  \E [ \Vert\widehat \lambda_{\kpart} - \breve \lambda_{\kpart}\Vert_\omega^2 \, \1_{\Xi_2^\complement}]  + 2 \, \E [ \Vert\breve \lambda_{\widetilde k} - \lambda\Vert_\omega^2  \1_{\Xi_2^\complement}] \\
&\leq 2  \E [ \Vert\widehat \lambda_{\Knmalpha} - \breve \lambda_{\Knmalpha}\Vert_\omega^2 \, \1_{\Xi_2^\complement} ]  + 2 \norm{\lambda}_\omega^2 \P(\Xi_2^\complement),
\end{align*}
and due to Assumption~\ref{ass:seq} and Lemma~\ref{l:prob:mho2c} it is easily seen that $\Vert \lambda \Vert^2_\omega \cdot \P(\Xi_2^\complement) \lesssim m^{-4}$.
Using the definition of $\Omega_j$, we further obtain
\begin{align}
\E [ \Vert\widehat \lambda_{\Knmalpha} - \breve \lambda_{\Knmalpha}\Vert_\omega^2 \, \1_{\Xi_2^\complement}]
&\leq 2m \sum_{0 \leq \abs j \leq \Knmalpha} \omega_j (  \E [\vert \widehat{\fou{\ell}}_j - \fou{\ell}_j \vert^4] ) ^{1/2} \P(\Xi_2^\complement)^{1/2}\notag\\
&\hspace{-1em}+ 2m \sum_{0 \leq \abs j \leq \Knmalpha} \omega_j \abs{\fou{\lambda}_j}^2 ( \E [\vert \widehat{\fou{f}}_j - \fou{f}_j \vert^4] )^{1/2} \P\left(\Xi_2^\complement \right)^{1/2}\notag\\
&\hspace{-6em}\lesssim m \P(\Xi_2^\complement)^{1/2} \sum_{0 \leq \vert j \vert \leq \Knmalpha} \frac{\omega_j}{n}  + \P(\Xi_2^\complement)^{1/2} \sum_{0 \leq \vert j \vert \leq \Knmalpha} \omega_j \vert \fou{\lambda}_j \vert^2, \label{eq:sums:mho2c}
\end{align}
where the last estimate follows by applying Theorem 2.10 from~\cite{petrov1995limit} with $p=4$ two times.
If $\Knmalpha=0$, Lemma~\ref{l:prob:mho2c} implies
\begin{equation*}
\E [ \Vert\widehat \lambda_{\Knmalpha} - \breve \lambda_{\Knmalpha}\Vert_\omega^2 \, \1_{\Xi_2^\complement}] \lesssim \frac{1}{nm} + \frac{1}{m^2}.
\end{equation*}
Otherwise, if $\Knmalpha >0$, we exploit $\omega_j \leq \omega_j^+\alpha_j^{-1}$, $\Knmalpha \leq N_n^\alpha$ and the definition of $N_n^\alpha$ to bound the first term in~\eqref{eq:sums:mho2c}. The second term in~\eqref{eq:sums:mho2c} can be bounded from above by noting that $\omega_j \leq \gamma_j$ thanks to Assumption~\ref{ass:seq}, and we obtain
\begin{equation*}
\E [ \Vert\widehat \lambda_{\Knmalpha} - \breve \lambda_{\Knmalpha}\Vert_\omega^2 \, \1_{\Xi_2^\complement}] \lesssim \frac{m \P(\Xi_2^\complement)^{1/2}}{\log(n+3)} \left( \sum_{0 \leq \vert j \vert \leq N_n^\alpha} \frac{1}{2\vert j \vert +1}\right)  +  \P(\Xi_2^\complement)^{1/2}.
\end{equation*}
Thanks to the logarithmic increase of the harmonic series, $N_n^\alpha \leq n$ and Lemma~\ref{l:prob:mho2c}, the last estimate implies
\begin{equation*}
\E [ \Vert\widehat \lambda_{\Knmalpha} - \breve \lambda_{\Knmalpha} \Vert_\omega^2 \, \1_{\Xi_2^\complement}] \lesssim \frac{1}{m} + \frac{1}{m^2},
\end{equation*}
if $\Knmalpha>0$, and thus
$\E [ \Vert\widehat \lambda_{\Knmalpha} - \breve \lambda_{\Knmalpha}\Vert_\omega^2 \, \1_{\Xi_2^\complement}] \lesssim \frac{1}{m} + \frac{1}{m^2}$,
independent of the actual value of $\Knmalpha$.
Using the obtained estimates, we conclude
\begin{equation*}\label{eq:part:second}
\E [ \Vert\widehat \lambda_\kpart - \lambda \Vert_\omega^2 \, \1_{\Xi_2^\complement}] \lesssim \frac{1}{m}.
\end{equation*}

\noindent \emph{Uniform upper bound for $\square_3$}: In order to find a uniform upper bound for $\square_3$, first recall the definition $\breve \lambda_k = \sum_{0 \leq \jabs \leq k} \fou{\lambda}_j \1_{\Omega_j}\e_j$ and consider the estimate
\begin{equation}\label{eq:square3:dec}
\E [\Vert \widehat \lambda_{\kpart} - \lambda \Vert_\omega^2 \1_{\Xi_1^\complement \cap \Xi_2}] \leq 2 \E [\Vert \widehat \lambda_\kpart - \breve \lambda_\kpart \Vert_\omega^2 \1_{\Xi_1^\complement \cap \Xi_2}] + 2 \E [ \Vert \breve \lambda_\kpart - \lambda \Vert_\omega^2  \1_{\Xi_1^\complement \cap \Xi_2}].
\end{equation}
Using the estimate $\Vert \breve \lambda_\kpart - \lambda \Vert_\omega^2 \leq \Vert \lambda \Vert_\omega^2$, we obtain for $\lambda \in \Lambda_\gamma^r$  by means of Lemma~\ref{l:prob:mho1c} that
\begin{equation*}
\E [ \Vert \breve \lambda_\kpart - \lambda \Vert_\omega^2  \1_{\Xi_1^\complement \cap \Xi_2}] \leq r \P(\Xi_1^\complement) \lesssim \frac{1}{n}
\end{equation*}
which controls the second term on the right-hand side of~\eqref{eq:square3:dec}.
We now bound the first term on the right-hand side of~\eqref{eq:square3:dec}.
If $\Knmalpha = 0$, we have $\kpart = 0$, and by means of the Cauchy-Schwarz inequality and Theorem~2.10 from~\cite{petrov1995limit} it is easily seen that
\begin{equation*}
\E [\Vert \widehat \lambda_\kpart - \breve \lambda_\kpart \Vert_\omega^2 \1_{\Xi_1^\complement \cap \Xi_2}] \lesssim \frac{1}{n}.
\end{equation*}
Otherwise, $\Knmalpha > 0$, and we need the following further estimate, which is easily verified:
\begin{align}\label{eq:square3:3sums}
\E [ \Vert \widehat \lambda_\kpart - \breve \lambda_\kpart \Vert_\omega^2 \1_{\Xi_1^\complement \cap \Xi_2} ] &\leq 3 \sum_{0 \leq \vert j \vert \leq \Knmalpha} \omega_j \E [ \vert \fou{\ell}_j/\widehat{\fou{f}}_j - \fou{\ell}_j / \fou{f}_j \vert^2 \1_{\Xi_1^\complement \cap \Xi_2} ]\notag\\
&\hspace{1em} + 3 \sum_{0 \leq \vert j \vert \leq \Knmalpha} \omega_j \E [  \vert \widehat{\fou{\ell}}_j - \fou{\ell}_j\vert^2 / \vert \fou{f}_j \vert^2 \1_{\Xi_1^\complement \cap \Xi_2} ]\notag\\
&\hspace{-4em} + 3 \sum_{0 \leq \vert j \vert \leq \Knmalpha} \omega_j \E [ \vert \widehat{\fou{\ell}}_j - \fou{\ell}_j \vert^2 \cdot \vert 1/\widehat{\fou{f}}_j - 1/\fou{f}_j\vert^2 \1_{\Xi_1^\complement \cap \Xi_2} ].
\end{align}
We start by bounding the first term on the right-hand side of~\eqref{eq:square3:3sums}. Using the definition of $\Xi_2$ and $\omega_j \leq \gamma_j$, we obtain for all $\lambda \in \Lambda_\gamma^r$ that
\begin{equation*}
\sum_{0 \leq \vert j \vert \leq \Knmalpha} \omega_j \E [ \vert \fou{\ell}_j/\widehat{\fou{f}}_j - \fou{\ell}_j / \fou{f}_j \vert^2 \1_{\Xi_1^\complement \cap \Xi_2} ] \leq \frac{r}{4} \cdot \P(\Xi_1^\complement) \lesssim \frac{1}{n}.
\end{equation*}
Since $\vert \fou{f}_j \vert^{-2} \leq d\alpha_j$ for $f \in \Fc_\alpha^d$, the Cauchy-Schwarz inequality in combination with Theorem~2.10 from~\cite{petrov1995limit} implies for the second term on the right-hand side of~\eqref{eq:square3:3sums} that
\begin{equation*}
\sum_{0 \leq \vert j \vert \leq \Knmalpha} \omega_j \E [  \vert \widehat{\fou{\ell}}_j - \fou{\ell}_j\vert^2 / \vert \fou{f}_j \vert^2 \1_{\Xi_1^\complement \cap \Xi_2} ] \lesssim \P(\Xi_1^\complement)^{1/2} \sum_{0 \leq \vert j \vert \leq \Knmalpha} \frac{\omega_j^+}{n\alpha_j}.
\end{equation*}
We exploit the definition of $N_n^\alpha$ together with $\Knmalpha \leq N_n^\alpha$ to obtain
\begin{equation*}
\sum_{0 \leq \vert j \vert \leq \Knmalpha} \omega_j \E [  \vert \widehat{\fou{\ell}}_j - \fou{\ell}_j\vert^2 / \vert \fou{f}_j \vert^2 \1_{\Xi_1^\complement \cap \Xi_2} ] \lesssim \frac{\P(\Xi_1^\complement)^{1/2}}{\log(n+3)} \sum_{0 \leq \vert j \vert \leq N_n^\alpha} \frac{1}{2 \vert j \vert + 1},
\end{equation*}
from which by the logarithmic increase of the harmonic series and Lemma~\ref{l:prob:mho1c} we conclude that
\begin{equation*}
\sum_{0 \leq \vert j \vert \leq \Knmalpha} \omega_j \E [  \vert \widehat{\fou{\ell}}_j - \fou{\ell}_j\vert^2 / \vert \fou{f}_j \vert^2 \1_{\Xi_1^\complement \cap \Xi_2} ] \lesssim \frac{1}{n},
\end{equation*}
independent of the actual value of $\Knmalpha$.
Finally, the third and last term on the right-hand side of~\eqref{eq:square3:3sums} can be bounded from above the same way after exploiting the definition of $\Xi_2$, and we obtain
\begin{equation*}
\sum_{0 \leq \vert j \vert \leq \Knmalpha} \omega_j \E [ \vert \widehat{\fou{\ell}}_j - \fou{\ell}_j \vert^2 \cdot \vert 1/\widehat{\fou{f}}_j - 1/\fou{f}_j\vert^2 \1_{\Xi_1^\complement \cap \Xi_2} ] \lesssim \frac{1}{n}.
\end{equation*}
Putting together the derived estimates, we obtain
\begin{equation*}
\E [\Vert \widehat \lambda_{\kpart} - \lambda \Vert^2 \1_{\Xi_1^\complement \cap \Xi_2}] \lesssim \frac{1}{n}.
\end{equation*}
The statement of the theorem follows by combining the upper  bounds for $\square_1$, $\square_2$, and $\square_3$.\qed

\subsection{Proof of Theorem~\ref{THM:FULLY:ADAP}}

Consider the event
\begin{equation*}\label{eq:def:mho2}
\Xi_3 \defeq \{ \Nalphamn \wedge \Malphamm \leq \Khatnm \leq \Nalphapn \wedge \Malphapm \}
\end{equation*}
in addition to the event $\Xi_1$ introduced in the proof of Theorem~\ref{THM:PART:ADAP} and the slightly redefined event $\Xi_2$ defined as
\begin{equation*}
\Xi_2 = \{ \forall 0 \leq \vert j \vert \leq \Malphapm : \vert 1 / \widehat{\fou{f}}_j - 1/ \fou{f}_j \vert \leq 1/(2 \vert \fou{f}_j\vert ) \text{ and } \vert \widehat{\fou{f}}_j \vert \geq 1/m \}.
\end{equation*}
Defining $\Xi = \Xi_1 \cap \Xi_2 \cap \Xi_3$, the identity $1=\1_{\Xi} + \1_{\Xi_2^\complement} + \1_{ \Xi_1^\complement \cap \Xi_2} + \1_{\Xi_1 \cap \Xi_2 \cap \Xi_3^\complement}$ motivates the decomposition
\begin{align*}
\E [ \Vert \widehat \lambda_\kfull - \lambda \Vert_\omega^2 ] &= \E [ \Vert \widehat \lambda_\kfull - \lambda \Vert_\omega^2 \1_{\Xi} ] + \E [ \Vert \widehat \lambda_\kfull - \lambda \Vert_\omega^2 \1_{\Xi_2^\complement} ]+ \E [ \Vert \widehat \lambda_\kfull - \lambda \Vert_\omega^2 \1_{\Xi_1^\complement  \cap \Xi_2} ]\\
&\hspace{+1em} + \E [ \Vert \widehat \lambda_\kfull - \lambda \Vert_\omega^2 \1_{\Xi_1 \cap \Xi_2 \cap \Xi_3^\complement}] \eqdef \square_1 + \square_2 + \square_3 + \square_4
\end{align*}
and we establish uniform upper risk bounds for the four terms on the right-hand side separately.

\noindent \emph{Uniform upper bound for $\square_1$}:
On $\Xi$ we have the estimate
$\frac{1}{4} \Delta_k \leq \widehat \Delta_k \leq \frac{9}{4} \Delta_k$, and thus
$$1/4 \left[\Delta_k \vee (k+4) \right] \leq  \widehat \Delta_k \vee (k+4)  \leq 9/4 \left[\Delta_k \vee (k+4) \right]$$
for all $k \in \{  0,\ldots, \Malphapm \}$.
This last estimate implies
\begin{align*}
\frac{2k+1}{4} &\Delta_k \frac{\log(\Delta_k \vee (k+4))}{\log(k+4)} \left(1 - \frac{\log 4}{\log(k+4)} \frac{\log(k+4)}{\log(\Delta_k \vee (k+4))} \right) \leq \widehat \delta_k\\
&\hspace{-2em}\leq \frac{9(2k+1)}{4} \Delta_k \frac{\log(\Delta_k \vee (k+4))}{\log(k+4)} \left(1+\frac{\log(9/4)}{\log(k+4)} \frac{\log(k+4)}{\log(\Delta_k \vee (k+4))} \right),
\end{align*}
from which we conclude $\frac{3}{100} \cdot \delta_k \leq \widehat \delta_k \leq \frac{17}{5} \cdot \delta_k$.
Putting $\pen_k = \frac{165}{2} (\widehat{\fou{\ell}}_0 \vee 1) \cdot \frac{\delta_k}{n}$, we observe that on $\Xi_2$ the estimate
\begin{equation*}
\pen_k \leq \hatpen_k \leq  340/3 \cdot \pen_k
\end{equation*}
holds for all $k \in  \{0,\ldots,\Malphapm \}$.
Note that on $\Xi$ we have $\widehat k \leq \Malphapm$ which using $\pen_{k \vee \widehat k} \leq \pen_k + \pen_{\widehat k}$ implies
\begin{align}
( \pen_{k \vee \widehat k} + \hatpen_k - \hatpen_{\widehat k}) \1_\Xi \leq 343/3 \cdot \pen_k \cdot \1_{\Xi} \label{eq:est:pen}.
\end{align}
Now, we can proceed by mimicking the derivation of~\eqref{eq:term:mho1} in the proof of Theorem~\ref{THM:PART:ADAP}.
More precisely, replacing the penalty term $\tildepen_k$ used in that proof by $\hatpen_k$, using the definition of $\pen_k$ above and~\eqref{eq:est:pen}, we obtain 
\begin{align*}
\E [ \Vert\widehat \lambda_{\widehat k} - \lambda\Vert_\omega^2 \, \1_{\Xi} ]  &\leq 7r \frac{\omega_k}{\gamma_k} + 40 \sum_{k=0}^{\Nalphapn} \E \hspace{-2pt}\left[\hspace{-3pt}\left( \sup_{t \in \Bc_k} \vert\langle \widetilde \Theta_{n \wedge m},t\rangle_\omega\vert^2- \frac{33(\fou{\ell}_0 \vee 1)\delta_k}{8n} \right)_+\right] \\
&\hspace{-1.5em}+ 32 \E[ \sup_{t \in \Bc_{\Knmalphap}} \vert \langle \check \Theta_{n \wedge m},t \rangle_\omega \vert^2]  + 4  \E [ (\pen_{k \vee \widehat k} + \hatpen_k - \hatpen_{\widehat k} ) \1_\Xi]   \\
&\hspace{-2em}\leq 7r\omega_k \gamma_k^{-1} + 40 \sum_{k=0}^{\Nalphapn} \E \left[ \left( \sup_{t \in \Bc_k} \vert \langle \widetilde \Theta_{n \wedge m},t \rangle_\omega \vert^2 - \frac{33(\fou{\ell}_0 \vee 1)\delta_k}{8n} \right)_+\right]   \\
&\hspace{1em}+ 32 \E [ \sup_{t \in \Bc_{\Knmalphap}} \vert\langle \check \Theta_{n \wedge m},t \rangle_\omega\vert^2 ]  + \frac{1372}{3} \, \pen_k.
\end{align*}
As in the proof of Theorem~\ref{THM:PART:ADAP}, the second and the third term are bounded applying Lemmata~\ref{l:ex:conc} (with $\delta_k^* \defeq \delta_k$ and $\Delta_k^* \defeq \Delta $) and~\ref{l:B4}, respectively.
Hence, by means of an obvious adaption of Statement~\ref{it:l:B2:a} in Lemma~\ref{l:B2} (with $N_n^\alpha$ replaced by $\Nalphapn$) and the estimates
\begin{equation*}
\Delta_k \leq d \Delta_k^\alpha, \quad \delta_k \leq d \zeta_d \delta_k^\alpha, \quad \frac{\delta_k}{\Delta_k} \geq 2k \zeta_d^{-1} \frac{\log(\Delta_k^\alpha \vee (k+4))}{\log(k+4)}
\end{equation*}
with $\zeta_d=\log(4d)/\log(4)$, we obtain in analogy to the way of proceeding in the proof of Theorem~\ref{THM:PART:ADAP} that
\begin{align}
\sup_{\lambda \in \Lambda_\gamma^r} \sup_{f \in \Fc_\alpha^d} \E [ \Vert \widehat \lambda_{\widehat k} - \lambda\Vert_\omega^2 \, \1_{\Xi} ] 
\lesssim \min_{0 \leq k \leq \Knmalpham} \max \left\lbrace \frac{\omega_k}{\gamma_k}, \frac{\delta_k^\alpha}{n} \right\rbrace  +  \Phi_m + \frac{1}{n}\label{eq:full:first}.
\end{align}

\noindent \emph{Upper bound for $\square_2$}: The uniform upper bound for $\square_2$ can be derived in analogy to the bound for $\square_2$ in the proof of Theorem~\ref{THM:PART:ADAP} using Assumption~\ref{ass:adap:fully} instead of Statement~\ref{it:l:B2:b} from Lemma~\ref{l:B2} in the proof of Lemma~\ref{l:prob:mho2c}.
Hence, we obtain
\begin{equation}\label{eq:full:second}
\sup_{\lambda \in \Lambda_\gamma^r} \sup_{f \in \Fc_\alpha^d} \E [ \Vert\widehat \lambda_\kpart - \lambda \Vert_\omega^2 \, \1_{\Xi_2^\complement}] \lesssim \frac{1}{m}.
\end{equation}

\noindent \emph{Upper bound for $\square_3$}:
The term $\square_3$ is bounded analogously to the bound established for $\square_3$ in the proof of Theorem~\ref{THM:PART:ADAP} (here, we do not have to exploit the additional Assumption~\ref{ass:adap:fully}), and we get
\begin{equation}\label{eq:full:third}
\sup_{\lambda \in \Lambda_\gamma^r} \sup_{f \in \Fc_\alpha^d} \E [\Vert \widehat \lambda_{\kpart} - \lambda \Vert^2 \1_{\Xi_1^\complement \cap \Xi_2}] \lesssim \frac{1}{n}.
\end{equation}

\noindent \emph{Upper bound for $\square_4$}: To find a uniform upper bound for the term $\square_4$, one can use exactly the same decompositions as in the proof of the uniform upper bound for $\square_3$ in Theorem~\ref{THM:PART:ADAP} by replacing the probability of $\Xi_1^\complement$ with the one of $\Xi_3^\complement$.
Doing this, we obtain by means of Lemma~\ref{l:prob:mho3c} that
\begin{equation}\label{eq:full:fourth}
\sup_{\lambda \in \Lambda_\gamma^r} \sup_{f \in \Fc_\alpha^d} 	\E [\Vert \widehat \lambda_{\kpart} - \lambda \Vert^2 \1_{\Xi_1 \cap \Xi_2 \cap \Xi_3^\complement}] \lesssim \frac{1}{m}.
\end{equation}
The result of the theorem now follows by combining~\eqref{eq:full:first},~\eqref{eq:full:second},~\eqref{eq:full:third} and~\eqref{eq:full:fourth}. \qed

\subsection{Auxiliary results}\label{s:aux:adap}

\begin{lem}\label{l:B2}
	Let Assumption~\ref{ass:seq} hold. Then the following assertions hold true.
	\begin{enumerate}[label=\alph*)]
		\item\label{it:l:B2:a} $\delta_{j}^\alpha/n \leq 1$ for all $n \in \N$ and $j \in \{ 0, \ldots, N_n^\alpha \}$,
		\item\label{it:l:B2:b} $\exp\left( - m \alpha_{M_m^\alpha}/(128d) \right) \leq C(d) m^{-5}$ for all $m \in \N$, and
		\item\label{it:l:B2:c} $\min_{1 \leq j \leq M_m^\alpha} \abs{\fou{f}_j}^2 \geq 2 m^{-1}$ for all $m \in \N$.
	\end{enumerate}
\end{lem}

\begin{proof}
	\ref{it:l:B2:a} In case $N_n^\alpha = 0$, we have $\delta^\alpha_{N_n^\alpha}=1$ and there is nothing to show.
	Otherwise $0 < N_n^\alpha \leq n$, and by definition of $N_n^\alpha$ we have $(2j+1) \Delta^\alpha_{j} \leq n/\log(n+3)$ for $0 \leq j \leq N_n^\alpha$ which by the definition of $\delta_{j}^\alpha$ implies that
	\begin{equation*}
	\delta_{j}^\alpha \leq \frac{n}{\log(n+3)} \cdot \frac{\log (n/((2j+1) \log(n+3)) \vee (j + 3))}{\log(j +3)}.
	\end{equation*}
	We consider two cases: In the first case, $n/((2j+1) \log(n+3)) \vee (j +3) = j +3$. Then $n \geq 1$ directly implies the estimate $\delta_{j}^\alpha \leq n$.
	In the second case, we have $n/((2j+1) \log(n+3)) \vee (j+3) = n/((2j+1) \log(n+3))$ and therefrom
	\begin{equation*}
	\delta_{j}^\alpha \leq n \log(n)/(\log(n+3) \log(j +3)) \leq n,
	\end{equation*}
	and thus $\delta_{j}^\alpha \leq n$ in both cases. Division by $n$ yields the assertion of the lemma.
	\ref{it:l:B2:b} Note that, due to Assumption~\ref{ass:seq}, we have $M_m^\alpha > 0$ for all sufficiently large $m$ and that it is sufficient to show the desired inequality for such values of $m$.
	By the definition of $M_m^\alpha$, we have $\alpha_{M_m^\alpha} \geq 640d m^{-1} \cdot \log (m+1)$ which implies
	$\exp(-m \alpha_{M_m^\alpha}/(128d) ) \leq \exp(-5 \log m) = m^{-5}$,
	and the assertion follows.
	\ref{it:l:B2:c} Take note of the observation that
	\begin{equation*}
	\min_{1 \leq j \leq M_m^\alpha} \abs{\fou{f}_j}^2 \geq \min_{1 \leq j \leq M_m^\alpha} \frac{\alpha_j}{d} = \frac{\alpha_{M_m^\alpha}}{d} \geq 640m^{-1}\cdot \log (m+1)
	\end{equation*}
	and $640m^{-1}\cdot \log (m+1) \geq 2m^{-1}$ for all $m \geq 1$.
\end{proof}

\begin{lem}\label{l:ex:conc}
	Let $(\delta^*_k)_{k \in \N_0}$ and $(\Delta^*_k)_{k \in \N_0}$ be sequences such that for all $k \geq 1$,
	\begin{equation*}
	\delta_k^* \geq \sum_{0 \leq \jabs \leq k} \frac{\omega_j}{\vert \fou{f}_j \vert^2} \qquad \text{and} \qquad \Delta_k^* \geq \max_{0 \leq \vert j \vert \leq k} \frac{\omega_j}{\vert \fou{f}_j \vert^2}.
	\end{equation*}
	Then, for any $k \in \{1,\ldots,n \wedge m\}$, we have
	\begin{align*}
	\E \left[ \left( \sup_{t \in \Bc_{k}} \vert  \langle \widetilde \Theta_{n \wedge m}, t  \rangle \vert^2 -  \frac{33 \delta_k^* (\fou{\ell}_0 \vee 1) }{8n} \right)_+\right] &\\ &\hspace{-15em} \leq K_1 \left\lbrace  \frac{\norm{f} \norm{\lambda} \Delta_k^\ast}{n} \exp \left( - K_2 \cdot \frac{\delta_k^\ast}{\norm{f} \norm{\lambda} \Delta_k^\ast } \right)  + \frac{\delta_k^\ast}{n^2}  \exp \left( -K_3 \sqrt{n} \right) \right\rbrace,
	\end{align*}
	with positive numerical constants $K_1$, $K_2$, and $K_3$.
\end{lem}

\begin{proof}
	The proof is a combination of the proofs of Lemma~A.1 in~\cite{kroll2016concentration} (which deals with the case $\omega \equiv 1$) and Lemma~A.4 in~\cite{johannes2013adaptive}.
	More precisely, one can apply Proposition~C.1 in~\cite{kroll2016concentration} with $c(\epsilon)$ from that statement replaced with $c(\epsilon)=4(1+2\epsilon)$ (this makes the proposition applicable also for complex-valued functions), $M_1^2=\delta_k^\ast$, $H^2 = \frac{\delta_k^\ast}{n} (\fou{\ell}_0 \vee 1)$, $\upsilon \defeq \Vert \lambda \Vert \Vert f\Vert \Delta_k^\ast (\fou{\ell}_0 \vee 1)$ and setting $\epsilon = 1/64$.
\end{proof}

\begin{lem}\label{l:B4}
	Let $m \in \N$, $k \in \Nzero$. Then
	$$\sup_{\lambda \in \Lambda_\gamma^r} \E [ \sup_{t \in \Bc_k} \vert\langle \check \Theta_{n \wedge m},t \rangle_\omega\vert^2 ] \leq C(d,r) \cdot \Phi_m.$$
\end{lem}

The proof follows along the lines of the proof of Lemma~A5 in~\cite{johannes2013adaptive} and is thus omitted.

\begin{lem}\label{l:prob:mho1c}
	Let Assumption~\ref{ass:seq} hold and consider the event $\Xi_1$ defined in Theorem~\ref{THM:PART:ADAP}.
	Then, for any $n \in \N$, $\P(\Xi_1^\complement) \leq 2 \exp(-Cn)$
	with a numerical constant $C>0$.
\end{lem}

\begin{proof}
	Note that
	$
	\P(\Xi_1^\complement) = \P(\widehat{\fou{\ell}}_0 \vee 1 < (\fou{\ell}_0 \vee 1)/2) + \P(\widehat{\fou{\ell}}_0 \vee 1 > 2 (\fou{\ell}_0 \vee 1)),
	$
	and the two terms on the right-hand side can be bounded by Chernoff bounds for Poisson distributed random variables (see~\cite{mitzenmacher2017probability}, Theorem~5.4) which yields the result.
\end{proof}

\begin{lem}\label{l:prob:mho2c}
	Let Assumption~\ref{ass:seq} hold and consider the event $\Xi_2$ defined in the proof of Theorem~\ref{THM:PART:ADAP}.
	Then, for any $m \in \N$, $\P(\Xi_2^\complement) \leq C(d) m^{-4}$.
\end{lem}

The proof follows along the lines of the proof of Lemma~A6 in~\cite{johannes2013adaptive} and is thus omitted.

\begin{lem}\label{l:prob:mho3c}
	Let Assumptions~\ref{ass:seq} and~\ref{ass:adap:fully} hold.
	The event $\Xi_3$ defined in~\eqref{eq:def:mho2} satisfies $\P(\Xi_3^\complement) \leq C(\alpha, d) m^{-4}$ for all $m \in \N$.
\end{lem}
The proof follows along the lines of the proof of Lemma~A7 in~\cite{johannes2013adaptive} and is thus omitted.

\printbibliography

\end{document}